\documentclass[a4paper,12pt]{article}
\usepackage{latexsym}
\usepackage{amsmath}
\usepackage{amssymb}
\usepackage{enumerate}
\usepackage{graphicx}
\usepackage{subcaption}
\usepackage{color}
\usepackage{theorem}
\newtheorem{theorem}{Theorem}[section]
\newtheorem{proposition}[theorem]{Proposition}
\newtheorem{lemma}[theorem]{Lemma}
\newtheorem{corollary}[theorem]{Corollary}

\theorembodyfont{\rmfamily}
\newtheorem{proof}{\textmd{\textit{Proof.}}}

\newtheorem{remark}[theorem]{Remark}
\newtheorem{example}[theorem]{Example}
\newtheorem{definition}[theorem]{Definition}

\newtheorem{acknowledgement}{\textmd{\textit{Acknowledgements.}}}

\makeatletter

\@addtoreset{equation}{section}
\makeatother

\newcommand{\qedd}{\hfill \Box}

\newcommand{\R}{\ensuremath{\mathbb{R}}}
\newcommand{\Sph}{\ensuremath{\mathbb{S}}}

\title{The geometry of a Randers rotational surface
\footnote{
Mathematics Subject Classification (2010)\,:\,53C60, 53C22.}
}
\author{Rattanasak , Pakkinee CHITSAKUL, Sorin V. SABAU}
\date{}
\begin{document}

\maketitle

\begin{abstract} We study the behaviour of geodesics on a Randers rotational surface of revolution. The main tool is the extension of Clairaut relation from Riemannian case to the Randers case. Moreover, we consider the embedding problem of this surface in a Minkowski space as a hypersurface. Finally, we study the rays and poles as well as the structure of the cut locus of a Randers rotational surface of revolution of von Mangoldt type.

%show that our Randers surface of revolution can be embedded in a Minkowski space as hypersurface.
\end{abstract}
%%%%%%%%%%%%%%%%%%%%%%%%%%%%%%%%%%%%%%%%%%%%%%%%%%%%%%%%%%%%%%%%%%%%%%
%%%%%%%%%%%%%%%%%%%%%%%%%%%%%%%%%%%%%%%%%%%%%%%%%%%%%%%%%%%%%%%%%%%

\section{Introduction}

The differential geometry of Riemannian surfaces has been extensively developed and it is almost impossible to find a reference containing all results on this topic (see for example \cite{AT}, \cite{dC}, \cite{SST} and many other resources). However, the geometry of Finsler surfaces, except for local computations, has not have been developed at the same rate (see \cite{BCS}, \cite{ST}).

In the present paper we study the global geometry of an abstract surface of revolution homeomorphic to $\R^2$ endowed with a Finsler metric of Randers type. Finslerian Clairaut relation is our main tool. This is a first generalisation of this type of the geometry of a Riemannian surface of revolution, a  well understood topic. 
% that allows a detailed description of its geodesics and other global geometrical properties. 

\bigskip

We review some basic notions of Finsler geometry. 

In 1931, E. Zermelo studied the following problem (see \cite{C}):

{\it Suppose a ship sails the sea and a wind comes up. How must the captain steer the ship in order to reach a given destination in the shortest time?
}

The problem was solved by Zermelo himself for the Euclidean flat plane and by D. Bao, C. Robles and Z. Shen (\cite{BRS}) in the case when the sea is a Riemannian manifold $(M,h)$ under the assumption that the wind $W$ is a time-independent mild breeze, i.e. $h(W,W)<1$. In the case when $W$ is a time-independent wind, they have found out that the path minimizing travel-time are exactly the geodesics of a Randers metric
\begin{equation*}
F(x,y)=\alpha(x,y)+\beta(x,y)=\frac{\sqrt{\lambda\cdot |y|^2+W_0^2}}{\lambda}-\frac{W_0}{\lambda},
\end{equation*}
where $W=W^i\frac{\partial}{\partial x^i}$ is the wind velocity, $|y|^2=h(y,y)$, $\lambda=1-|W|^2$ and
$W_0=h(W,y)$. 

The Randers metric $F$ is said {\it to solve the Zermelo's navigation problem} in the case of a mild breeze. The condition $h(W,W)<1$ ensures that $F$ is a positive-definite Finsler metric. 
Moreover, it can be shown that a Randers space is of constant flag curvature if and only if the underlying Riemannian manifold $(M,h)$ is of constant sectional curvature and the wind $W$ is a Killing vector field of $h$ (see \cite{BRS}, \cite{BCS}).
The Zermelo's navigation approach was extended in \cite{YS} to Kropina metrics as well. Finally, we recall that the geometry of the sphere regarded as Randers surface of revolution with Killing wind was studied in detail (\cite{R}), but the more general case of a Randers surface of rotation,  of whose Riemannian sectional curvature is not constant, is studied in the present paper for the first time.

Our paper is two aimed. We intend to study the geometry of a Randers type metric on a surface of revolution by generalising the Clairaut relation to the Finslerian setting, as well as to illustrate the Zermelo's navigation process for a better understanding of it.

More precisely, we perturb the induced canonical Riemannian metric $h$ of a surface of revolution by the rotational vector field $W$ obtaining in this way a Randers type metric on $M$ through the Zermelo's navigation process. We study some of the local and global geometrical properties of the geodesics on the surface of revolution $M$ endowed with this Randers metric. 

Here are our main results.

\begin{theorem}\label{thm:global Finsler geodesics}
%If $(M,h)$ is a surface of revolution whose profile curve is the bounded %function $x=m(r)<\frac{1}{\mu}$ and $W$ is the breeze on $M$ blowing along parallels,
%we consider 

Let $(M,F=\alpha+\beta)$ be
the rotational Randers metric constructed from the navigation data $(h,W)$, where $(M,h)$ is a Riemannian surface of revolution whose warp function is bounded $m(r)<\frac{1}{\mu}$, $\mu>0$, and $W=\mu\frac{\partial}{\partial \theta}$ is the breeze on $M$ blowing along parallels, then the unit speed Finslerian geodesics $\mathcal{P}:(-\epsilon,\epsilon)\to M$ are given by
\begin{equation}\label{global Finsler geodesics}
\mathcal{P}(s)=(r(s),\theta(s)+{\mu}s),
\end{equation}
where $\gamma(s)=(r(s),\theta(s))$ is a $h$-unit speed geodesic.
\end{theorem}

 Unlike Riemannian manifolds, Finsler manifolds cannot always be isometrically embedded in a sufficiently higher dimensional Minkowski space (\cite{Sh}). However, this is possible in the present case.

\begin{theorem}\label{embedding}
The rotational Randers space $(M,F=\alpha+\beta)$ can be isometrically embedded into the Minkowski space $(\mathcal{U}_\mu,\tilde{F})$ if and only if the Riemannian surface of revolution $(M,h)$ can be isometrically embedded in $(\R^3,\delta)$. 
\end{theorem}

The geometry of a Riemannian surface of revolution is completely governed by the Clairaut relation (see \cite{SST}), but the correspondent of this relation in Finsler geometry is unknown. We give here a generalisation of the Riemannian Clairaut relation to the case of a Randers rotational surface of revolution. 
%One can see that this is a natural generalization of the Riemannian Clairaut %relation.

\begin{theorem}\label{Finslerian Clairaut relation}
Let $\gamma(s)=(r(s),\theta(s))$ be an $h$-geodesic of Clairaut constant $\nu$, that makes an angle $\phi(s)$ with the profile curve passing through $\gamma(s)$,
and let $\mathcal P(s)$ be the corresponding $F$-geodesic on the Randers rotational surface of revolution $(M,F)$. Then the following relations hold good.

\begin{equation}\label{Clairaut Fins}
\sqrt{ 1+2\mu \nu+\mu^2m^2}\cos(\psi-\phi)=1+\mu\nu,
\end{equation}

\begin{equation}\label{Clairaut Fins 1}
m\sin\psi=\frac{\nu+\mu m^2}{\sqrt{ 1+2\mu \nu+\mu^2m^2}},
\end{equation}

where $\psi$ is the angle between $\dot{\mathcal P}(s)$ and the profile curve passing through $\mathcal P(s)$. 

\end{theorem}

Obviously, these two forms of the Clairaut relation are equivalent and they reduce to the classical Clairaut relation when $F$ is Riemannian. 

The geometry of geodesics of $(M,F)$ can now be easily obtained using these relations (see Section \ref{sec: geodesics behaviour}). We mention here a result about the set of poles of a Randers rotational metric (see Section \ref{sec: geodesics behaviour} for definitions). 

\begin{theorem}\label{thm:poles}
For any point $q\neq p$,  let $\gamma$ be 
a geodesic from $q$, which is not tangent to the twisted meridian through $q$. Then $\gamma$
cannot be a ray,
 that is the vertex $p$ is the unique pole of $(M,F)$.
\end{theorem}

The cut locus of a point $q$ in a Riemannian or Finsler manifold is, roughly speaking,
the set of all other points for which there are multiple minimizing geodesics connecting
them from $q$. In Section \ref{sec: von Mangoldt} we define the notion of Finsler von Mangoldt  surface of revolution and determine the structure of the cut locus of a point in a rotational Randers von Mangoldt  surface of revolution (see \cite{SST}, \cite{T} for the Riemannian case and \cite{SaT} for the general Finsler case). 

\begin{theorem}\label{thm: Randers cut locus}
Let $(M,F=\alpha+\beta)$ be a rotational Randers von Mangoldt  surface of revolution. Then, for any point $q\neq p$, the Finslerian cut locus $\mathcal{C}^{(F)}_q$ of $q$ is the Jordan arc 
\begin{equation*}
\mathcal{C}^{(F)}_q=\{\varphi(s,\tau_q(s)):s\in {[c,\infty)}\},
\end{equation*}
where $\varphi(c,\tau_q(c))$ is the first conjugate point of $q$ along the twisted meridian $\varphi(s,\tau_q(s))$.
\end{theorem}

\begin{acknowledgement}
We express our gratitude to M. Tanaka for pointing out some errors in the initial version of  the paper. %We also thank to H. Shimada and  R. Yoshikawa for useful comments that have improved the quality of the paper. 
\end{acknowledgement}
%%%%%%%%%%%%%%%%%%%%%%%%%%%%%%%%%%%%%%%%%%%%%%%%%%%%%%%%%%%%%%%%%%%%%
%%%%%%%%%%%%%%%%%%%%%%%%%%%%%%%%%%%%%%%%%%%%%%%%%%%%%%%%%%%%%%%%%%%%%
\section{A rotational surface of revolution}

\subsection{The geometry of a Riemannian surface of revolution}

A {\it Riemannian (abstract) surface of revolution} is a complete Riemannian manifold $(M,h)$ homeomorphic to $\R^2$ that admits a point $p\in M$ such that the Gaussian curvature $G$ of $h$ is constant on each geodesic circle $\{x\in M : d_h(p,x)=\rho\}\subset M$, for any radius $\rho>0$. The point $p$ is called {\it the vertex} of the surface of revolution $(M,h)$. 

\begin{remark}
It can be seen that $(M,h)$ is a surface of revolution if and only if for any two points $x$, $y\in M$, such that $d_h(p,x)=d_h(p,y)$, there exists a Riemannian isometry $\varphi:M\to M$ such that $\varphi(x)=y$. One can consider this property as the definition of a Riemannian (abstract) surface of revolution.
\end{remark}

It is known (see  \cite{T}, \cite{SST}) that the surface or revolution $(M,h)$
can be endowed with the warped Riemannian metric 
\begin{equation}\label{Riemann metric on surf of revol}
ds^2=dr^2+m^2(r)d\theta^2,
\end{equation}
where $(r,\theta)\in [0,\infty)\times (0,2\pi]$ are the $h$-geodesic polar coordinates around $p$ on $M$, and $m:[0,\infty)\to [0,\infty)$, is a smooth odd function such that $m(0)=0$, $m'(0)=1$. 

\begin{remark}
The above definition is a natural generalisation of the classical 
Riemannian surface of revolution $M$ isometrically embedded in $\R^3$ (see \cite{dC}, \cite{SST}).
Indeed, for a positive function $f:[0,\infty)\to [0,\infty)$ one defines {\it a classical surface of revolution} 
\begin{equation}\label{the surf of revol M}
M:=\{(f(u)\cos v, f(u)\sin v, u)\in \R^3; \text{ }u\in [0,\infty),0<v\leq 2\pi\}
\end{equation}
by revolving 
the
 profile curve $x=f(z)$ around the $z$ axis. Clearly, $M$ is a surface homeomorphic to $\R^2$. 
 
 Abstract surfaces of revolution include surfaces that cannot be isometrically embedded in the Euclidean space $\R^3$ and surfaces whose profile curve cannot be written as $x=f(z)$. 
\end{remark}

Returning to the general case, recall that the equations of an $h$-unit speed geodesic $\gamma(s):=(r(s),\theta(s))$ of $(M,h)$ are
\begin{equation}\label{eq 4}
\begin{cases}
 \frac{d^2r}{ds^2}-mm'\left(\frac{d\theta}{ds}\right)^2=0\\ % \label{eq 3}\\
 \frac{d^2\theta}{ds^2}+2\frac{m'}{m}\frac{dr}{ds}\frac{d\theta}{ds}=0
\end{cases}
\end{equation}
with the unit speed parametrization condition 
\begin{equation}\label{eq 1}
\left(\frac{dr}{ds}\right)^2+m^2\left(\frac{d\theta}{ds}\right)^2 =1.
\end{equation}
It follows that every profile curve, or {\it meridian}, is an $h$-geodesic,
%Therefore $\dot{\gamma}$ is not tangent to any profile curve if $\dot{\gamma}(s_0)$ for some $s_0$ is not a tangent $(v(s_0)$ is not constant). 
%A natural question is which parallels are geodesics. It can be seen 
and that a parallel 
$\{r=r_0\}$ is geodesic if and only if $m'(r_0)=0$.

%The intersection point of $M$ with the $z$-axis is called the {\it vertex} of $M$. 
A point $p\in M$ is called {\it a pole} if any two $h$-geodesics from $p$ do not meet again. In other words, the cut locus of $p$ is empty. 
A unit speed geodesic of $(M,h)$ is called {\it a ray} if $d_h(\gamma(0),\gamma(s))=s$, for all $s\geq 0$.

We observe that  (\ref{eq 4}) implies
\begin{equation}\label{h-prime integral}
\frac{d\theta(s)}{ds}m^2(r(s)) = \nu=\text{ const},
\end{equation}
that is the quantity $\frac{d\theta}{ds}m^2$ is conserved along the $h$-geodesics. % (it is a prime integral).  

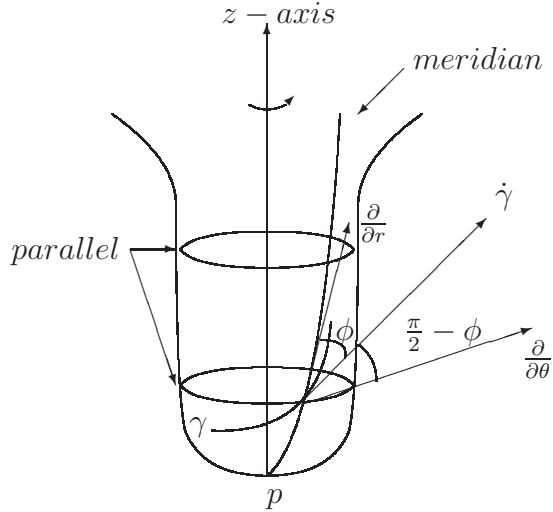
\begin{figure}[h]\label{fig 2}
\begin{center}
\setlength{\unitlength}{1.2cm}
\begin{picture}(6,5)

\put(3.5,0){\vector(0,1){5}}
\put(3.5,-0.3){$p$}
\qbezier(1.8,4)(2.5,3.5)(2.5,3)
\qbezier(2.5,3)(2.5,1)(2.6,0.5)

\qbezier(2.6,0.5)(2.8,0)(3.5,0)
\qbezier(4.4,0.5)(4.2,0)(3.5,0)

\qbezier(5.2,4)(4.5,3.5)(4.5,3)
\qbezier(4.5,3)(4.5,1)(4.4,0.5)

\qbezier(2.55,1)(2.7,0.8)(3.5,0.8)
\qbezier(3.5,0.8)(4.3,0.8)(4.45,1)
\qbezier(2.55,1)(2.7,1.2)(3.5,1.2)
\qbezier(3.5,1.2)(4.3,1.2)(4.45,1)

\qbezier(2.55,2.5)(2.7,2.3)(3.5,2.3)
\qbezier(3.5,2.3)(4.3,2.3)(4.45,2.5)
\qbezier(2.55,2.5)(2.7,2.7)(3.5,2.7)
\qbezier(3.5,2.7)(4.3,2.7)(4.45,2.5)

\qbezier(3.5,0)(4.1,0.5)(4.3,4)

\put(3.9,0.8){\vector(1,4){0.5}}
\put(3.9,0.8){\vector(3,1){2.5}}

\qbezier(2.9,0.5)(4.1,0.45)(4.2,1.7)

\qbezier(4.09,1.5)(4.4,1.5)(4.35,1.3)
\qbezier(4.5,1.45)(4.7,1.3)(4.7,1.05)

\put(4.25,1.5){$\phi$}
\put(5,1.5){$\frac{\pi}{2}-\phi$}

\put(4.5,2.7){$\frac{\partial}{\partial r}$}
\put(6.3,1.2){$\frac{\partial}{\partial \theta}$}
\put(6,3){$\dot{\gamma}$}
\put(2.65,0.5){$\gamma$}

\qbezier(3.3,4.1)(3.5,4)(3.7,4.1)

\put(3.7,4.1){\vector(1,1){0.1}}
\put(3.9,0.84){\vector(1,1){2}}
\put(5,4.5){\vector(-1,-1){0.5}}

\put(3,5){$z-axis$}
\put(5.1,4.5){$meridian$}

\put(2,2.5){\vector(1,0){0.5}}
\put(2,2.5){\vector(1,-3){0.5}}
\put(0.7,2.4){$parallel$}

\end{picture}

\end{center}
\caption{The angle $\phi$ between $\dot{\gamma}$ and a meridian for a classical surface of revolution.}
\end{figure} 

\begin{theorem}{\rm (Clairaut Relation)}\label{thm: h-Clairaut rel}
If $\gamma(s)=(r(s),\theta(s))$ is a geodesic on the surface of revolution $(M,h)$, then the angle $\phi(s)$ between $\dot{\gamma}$ and the profile curve passing through a point $\gamma(s)$ satisfy
%\begin{equation*}
 $m(r(s))\sin\phi(s)=\nu.$
 %where $\theta(s)$ the angle between the tangent vector $\dot{\gamma}(s)$ to the unit speed geodesic $\gamma$ and the meridian $v$= constant. 
%\end{equation*}
\end{theorem}

The constant $\nu$ is called the {\it Clairaut constant} and it plays an important role in the study of $h$- geodesics of $M$. Indeed, one can easily see that the Clairaut constant $\nu$ vanishes if and only if $\gamma$ is tangent to a meridian. Moreover, if the Clairaut constant  $\nu$ is non-vanishing, then $\gamma$ does not pass through the vertex of $M$.

\begin{lemma}{\rm (\cite{SST})}\label{paralells h-length}
We denote by $\mathcal L_h(r)$ the $h$-length of a parallel found at distance $r$ from the vertex $p$.
\begin{enumerate}  
\item If $\liminf_{r\to \infty}\mathcal L_h(r)=0$ then for any point $q\neq p$, the sub-ray 
$\mu_q|_{[d(p,q),\infty)}$ of the meridian $\mu_q$ from $p$ through $q$ is the unique ray emanating from $q$.
\item If $\int_1^{\infty}\mathcal L_h^{-2}(r)=\infty$ then for any point $q\neq p$, a geodesic $\gamma$ from $q$, which is not tangent to the meridian through $q$, cannot be a ray, that is the vertex $p$ is the unique pole of $(M,h)$.
\end{enumerate}
\end{lemma}

We recall here a remarkable class of Riemannian surfaces of revolution. A Riemannian surface of revolution with vertex $p$ is called {\it von Mangoldt surface} if, for any two points  $x,y\in M$ such that $d_h(p,x)\geq d_h(p,y)$ we have $G(x)\leq G(y)$, where $G$ is the Gauss curvature of $h$ (see \cite{T}, \cite{SST}). The cut locus structure of such a surface is determined in detail.

\begin{theorem}{\rm (\cite{T}, \cite{SST})}
If $(M,h)$ is a von Mangoldt surface with vertex $p$, then for any $q\in M$, $q\neq p$, the $h$-cut locus $\mathcal C_q^{(h)}$ of $q$ coincides to the sub-arc $\tau_q[t_0,\infty)$, where $\tau_q$ is the opposite meridian of the meridian $\mu_q$ from $p$ through $q$, and $\tau_q(t_0)$ is the first conjugate point of $q$ along $\tau_q$. 
\end{theorem}

\begin{remark}
The von Mangoldt surfaces are important in modern differential geometry not only for their computable cut locus, but also for Toponogov comparison theorems that use a von Mangoldt surface as model (\cite{KT}).

\end{remark}

% % % % % % % % % % % % % % % % % % % % % % % % % % % % %5
\subsection{A Rotational Randers metric}

Let $m:[0,\infty)\to [0,\infty)$, be a smooth odd function such that $m(0)=0$, $m'(0)=1$, and we consider the Riemannian surface of revolution $(M,h)$ as above. Furthermore, we assume that $m$ is  
bounded, i.e. there exists a constant $\mu>0$ such that $m(r)<\frac{1}{\mu}$ for all $r\geq 0$. % and construct the surface of revolution $M$ obtained by revolving this curve around $z$ axis. % As shown in the previous section $(M,h)$ is a complete Riemannian surface. 

We construct a rotational Randers metric on $M$ by putting
%\begin{equation*}
$W:={\mu}\cdot\frac{\partial}{\partial \theta}$
%\end{equation*}
that is, in the $h$-orthogonal coordinates system $(\frac{\partial}{\partial r},\frac{\partial}{\partial \theta})$ of $T_xM$ we have 
%\begin{equation*}
 $W=(W^1,W^2)=(0,{\mu})$.
%\end{equation*}

It follows
%\begin{equation*}
$h(W,W)=h\left({\mu}\cdot\frac{\partial}{\partial \theta},{\mu}\cdot\frac{\partial}{\partial \theta}\right)={\mu^2}\cdot h\left(\frac{\partial}{\partial \theta},\frac{\partial}{\partial \theta}\right)=\left({\mu}m\right)^2<1.$
%\end{equation*}

The navigation data $(h,W)$ gives new data
%\begin{equation} %\label{a and b}
$a_{ij}=\frac{\lambda\cdot h_{ij}+W_iW_j}{\lambda^2}$,  $b_i=-\frac{W_i}{\lambda}$
%\end{equation}
where $W_i=h_{ij}W^j$, $\lambda=1-h(W,W)=1-\mu^2m^2>0$.
We observe that 
%\begin{equation*}
%W_1=h_{11}W^1+h_{12}W^2=0 \text{, } W_2=h_{21}W^1+h_{22}W^2=\mu {f^2},
%\end{equation*}
%that is
%\begin{equation}
 $(W_1,W_2)=(0,\mu m^2).$ 
%\end{equation}

A simple computation shows that 
%It follows
%\begin{lemma}
%the Riemannian metric $(a_{ij})$ and the functions $(b_i)$ obtained through Zermelo's navigation process from $h$ and $W$ are 
\begin{equation}\label{eq 3.1}
(a_{ij})=\left(
\begin{array}{cc}
\frac{1}{1-{\mu}^2m^2} & 0\\
0 & \frac{m^2}{\left(1-{\mu}^2m^2\right)^2}
\end{array}\right)
\text{, }
b_i=\left(
\begin{array}{c}
0 \\
-\frac{\mu m^2}{1-{\mu}^2m^2}
\end{array}\right),\qquad i,j=1,2.
\end{equation}
%\end{lemma}

It is straightforward to see that
%\begin{equation*}
 $\alpha(b,b)=a^{ij}b_ib_j=h(W,W)=h_{ij}W^iW^j<1.$
%\end{equation*}

We obtain

\begin{proposition}
If $(M,h)$ is a surface of revolution whose profile curve is the bounded function $x=m(r)<\frac{1}{\mu}$ and $W$ is the breeze on $M$ blowing along parallels, then the Randers metric $(M,F=\alpha+\beta)$ obtained by the Zermelo's navigation process on $M$ is a Finsler metric 
on $M$, where $\alpha=\sqrt{a_{ij}(x)y^iy^j}$, $\beta=b_i(x)y^i$ are defined in (\ref{eq 3.1}).
\end{proposition}

We will call this Finsler metric the {\it rotational Randers metric on the surface of revolution $M$}. We point out that the assumption $m$ bounded is essential for the positive definiteness of $F$.  
This assumption combined with the Clairaut relation for $h$-geodesics %$f(u(s))\sin\theta(s)=\nu$ 
implies
%\begin{equation*}
$|\nu|=|m(r(s))|\cdot|\sin\phi(s)|\leq |m(r(s))|<\frac{1}{\mu}$, and therefore the Clairaut constant of the $h$-geodesics on $(M,h)$ must satisfy
%\begin{equation}
 $|\nu|<\frac{1}{\mu}.$
%\end{equation}

 An {\it isometry} of a Finsler manifold $(M,F)$ is a mapping
$\phi:M\to M$ that is diffeomorphism such that for any $x\in M$ and $X\in T_xM$, we have
%\begin{equation}
$F(\phi(x),\phi_{*,x}(X))=F(x,X).$
%\end{equation}
Equivalently, if we denote by $d_F$ the induced distance function of $F$ on $M$, then the isometry group of $(M,F)$ coincides with the isometry group of the quasi-metric space $(M,d_F)$, that is we have
%\begin{equation}
$d_F(\phi(x), \phi(y))=d_F(x,y)$,
%\end{equation}
for any points $x,y\in M$ (\cite{D}). 
The isometry group of $(M,F)$ is a Lie group of transformations on $M$. 

A smooth vector field $X$ on $M$ is called an $F$-{\it Killing vector field} if every local one-parameter transformation group $\phi_t$ of $M$ generated by $X$ consists of local isometries of $(M,F)$.

\begin{proposition}
\begin{enumerate}
\item The vector field $W={\mu}\frac{\partial}{\partial \theta}$ is a Killing vector field on the surface of revolution $M$ for the Riemannian structures $h$ and $a$, as well as for the Randers metric $F=\alpha+\beta$.
\item The compact Lie group $SO(2)$ acts by isometries on $(M,F)$, $(M,h)$ and $(M,a)$.
\end{enumerate}
\end{proposition}
\begin{proof}
1. 
Remark that the tangent map of the flow $\varphi$ of $W$ is actually the identity map of $T_xM$, for any $x=(r,\theta)\in M$, that is
%\begin{equation}
 $\varphi_{*,x}:T_xM\to T_{\varphi_s(x)}M,\quad \varphi_{*,x}(X)=X|_{\varphi_s(x)}.  $
%\end{equation}
Then the details follows direct from the definitions.

2. Remark that if we write the surface of revolution \eqref{the surf of revol M} as 
$
\Phi:M\to \mathbb C\times \R, 
\quad (r,\theta)\mapsto (m(r)e^{i\theta},r),
$
then we can define the action
$$
\xi:SO(2)\times M\to M,\quad (\alpha,p)\mapsto (m(r)e^{i(\theta+\alpha)},r),
$$
for any $p=(m(r)e^{i\theta},r)\in M$. We show that this action is by isometries, that is
$
\xi_\alpha: M\to M,\quad \xi_\alpha(p)=\xi(\alpha,p)
$
is an isometry for each of the three metrical structures on $M$, for any $\alpha\in \Sph^1=SO(2)$. 

Locally, on $M$, we can see that  $\xi_\alpha: M\to M$ actually is
$$
\xi_\alpha(p)=\xi_\alpha(\Phi(r,\theta))=\xi_\alpha(m(r)e^{i\theta},r)=(m(r)e^{i(\theta+\alpha)},r)=
\Phi(r,\theta+\alpha),
$$
that is, on $M$, we have $\xi_\alpha:(r,\theta)\mapsto (r,\theta+\alpha)$ and hence the tangent mapping
$
(\xi_\alpha)_{*,(r,\theta)}:T_{(r,\theta)}M\to T_{(r,\theta+\alpha)} M
$
is the identity map. Therefore, taking into account that functions $h_{ij}$, $a_{ij}$, $b_i$ are all depending on $u$ only, that is are all rotational invariant, the mapping $\xi_\alpha$ must be an isometry for the three metrical structures on $M$.$\qedd$
\end{proof}

We can prove now one important result.

\begin{proof}[Proof of Theorem \ref{thm:global Finsler geodesics}]
Recall that Zermelo navigation gives 
\begin{equation}
h(\dot{\gamma}(s),\dot{\gamma}(s))=1 \textrm{   if and only if    } 
F(\dot{\mathcal P}(s),\dot{\mathcal P}(s))=1. 
\end{equation}
Then the conclusion follows from \cite{R}, or can be verified directly.$\qedd$

\end{proof}

\begin{corollary}
The pair $(M,F)$ is a forward complete Finsler surface of Randers type. 
\end{corollary}

\begin{proof}
If $\gamma(s)=(r(s),\theta(s))$ is an $h$-geodesic that can be extended to infinity by taking $s\to \infty$, then the corresponding Finslerian geodesic $\mathcal P(s)=(r(s),\theta(s)+\mu s)$ can also be extended to infinity. Therefore, the completeness of the Riemannian metric $h$ implies the completeness of $F$.$\qedd$
\end{proof}

\begin{proposition}\label{prop: meeting on a parallel}
Let $q\in M$ be a point different from the vertex $p$ and assume $q=(r_0,0)$. Consider the parallel $\{r=r_0\}$ through $q$, $\gamma:[0,2\pi]\to M$, on $M$ and denote by $\gamma^+$ and $\gamma^-$ the same parallel traced  in the direction of $W$ and $-W$, respectively. 

Then there exists a point $\hat{q}$, different from $q$, on $\gamma|_{[0,2\pi]}$ such that
\begin{equation*}
\mathcal L _F(\gamma^+|_{q\hat{q}})=\mathcal  L _F(\gamma^-|_{q\hat{q}})=\pi \mu m(r_0),
\end{equation*}
where $\gamma^+|_{q\hat{q}}$ and  $\gamma^-|_{q\hat{q}}$ denote the arcs of $\gamma^+$ and $\gamma^-$ from $q$ to $\hat{q}$, respectively.
\end{proposition}
\begin{proof}
Since $\gamma$ is a parallel, we have $\dot{\gamma}=(0,1)$, $\dot{\gamma^+}=W=(0,\mu)$, $\dot{\gamma^-}=-W=(0,-\mu)$. 

For any $s_1,s_2\in [0,2\pi)$ the $F$-length of the sub-arcs $\gamma^+|_{[0,s_1]}$ and
 $\gamma^-|_{[0,s_2]}$, respectively, are
%By straightforward computation it follows
\begin{equation}\label{length +}
\begin{split}
\mathcal{L}_F({\gamma^+}|_{[0,s_1]})&=
\int_0^{s_1}F(\dot{\gamma}^+)ds
=\int_0^{s_1}\left[\sqrt{a_{22}(\gamma^+(s))(\dot{\gamma}^+(s))^2}+b_2(\gamma^+(s))
\dot{\gamma}^+(s)\right] ds\\
&=\int_0^{s_1}\left[\frac{m(r_0)}{1-\mu^2 m^2(r_0)}\cdot\mu-
\frac{\mu m^2(r_0)}{1-\mu^2m^2(r_0)}\cdot\mu\right]ds
%&=\left(\frac{\mu m(r_0)}{1-\mu^2m^2(r_0)}-\frac{\mu^2 %m^2}{1-\mu^2m^2}\right)_{r=r_0}\cdot\int_0^{2\pi} ds
=\frac{\mu m(r_0)}{1+\mu m(r_0)}s_1
\end{split}
\end{equation} 
and similarly 
\begin{equation}\label{length -}
\mathcal{L}_F({\gamma^-}|_{[0,s_2]})=\frac{\mu m(r_0)}{1-\mu m(r_0)}s_2.
\end{equation} 

Putting now conditions that two travellers on the parallel $\{r=r_0\}$ starting from $q$ tracing $\gamma^+$ and $\gamma^-$, respectively, meet on the way at the point $\hat{q}=\gamma^+(s_1)=\gamma^-(s_2)$, and that they travel equal lengths, we get
the linear system
\begin{equation*}
\begin{cases}
s_1+s_2=2\pi\\
\frac{s_1}{1+\mu m(r_0)}=\frac{s_2}{1-\mu m(r_0)}
\end{cases}
\end{equation*}
with the solution $(s_1,s_2)=(\pi(1+\mu m(r_0)),\pi(1-\mu m(r_0)) )$
and hence the conclusion follows.
$\qedd$
\end{proof}

\begin{corollary}\label{thm: 2 parallels length}
For each $r_0\in\R$ such that $m'(r_0)=0$, there exists two closed unit speed $F$-geodesics $\mathcal{P}^+$ and $\mathcal{P}^-$ on $M$, that trace the parallel $\{r=r_0\}$ in the direction of $W$ and $-W$,  of length
\begin{equation*}
\mathcal L _F^+(r_0)=\frac{\pi m(r_0)}{1+\mu\cdot m(r_0)} \ \textrm{ and } \ \mathcal L _F^-(r_0)=\frac{\pi m(r_0)}{1-\mu\cdot m(r_0)}, 
\end{equation*}
respectively.
\end{corollary}

\begin{proof}
It follows immediately from formulas \eqref{length +} and  \eqref{length -} by putting $s_1=2\pi$ and $s_2=2\pi$, respectively.
$\qedd$
\end{proof}

\begin{remark}
\begin{enumerate}
\item The $h$-length of 
the parallel $\{r=r_0\}$ 
 is $\mathcal L _h(r_0)=2\pi m(r_0)$.
 \item
 Remark that $\mathcal{L}^+_F(r_0)<\mathcal{L}_h(r_0)<\mathcal{L}^-_F(r_0)$. %for any $0<\mu<1$.
  This is constant with fundamental property of the solution of Zermelo navigation problem, namely that the $F$-geodesics deviated in the rotation direction are always shorter than $h$-geodesics. %Corollary \ref{f shorter than h} confirm that this is true for parallels that are geodesics.
 Nevertheless, in the case of Randers rotational surface or revolution, this is true for any parallel, regardless it is geodesic or not.
\item The number of closed $F$-geodesics on $M$ is double the number of closed $h$-geodesics.
\end{enumerate}
\end{remark}

\begin{corollary}\label{cor: F-parallels}
If the function $m$ has $n$ discrete critical points, then there exists at least $2n$ closed F-geodesics on $M$.
\end{corollary}

\begin{lemma}\label{F_von_lem1}
For any point $q\in M$ the $h$-distance and $F$-distance from $p$ to $q$ coincide, i.e. $d_F(p,q)=d_h(p,q)$.
\end{lemma}

%\bigskip
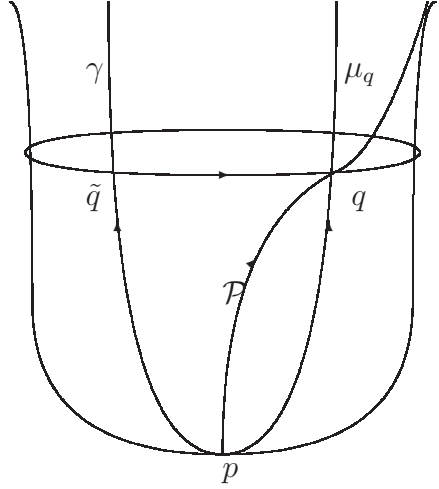
\begin{figure}[h]
\begin{center}
\setlength{\unitlength}{1cm}
\begin{picture}(7,7)

\qbezier(0.7,6)(1,6)(1,2)
\qbezier(1,2)(1,0)(3.5,0)

\qbezier(6.3,6)(6,6)(6,2)
\qbezier(6,2)(6,0)(3.5,0)

%parallel
\qbezier(0.9,4)(1,4.3)(3.5,4.3)
\qbezier(0.9,4)(1,3.7)(3.5,3.7)

\qbezier(6.1,4)(6,4.3)(3.5,4.3)
\qbezier(6.1,4)(6,3.7)(3.5,3.7)

%gamma
\qbezier(2,6)(2,0)(3.5,0)

%eta
\qbezier(5,6)(5,0)(3.5,0)

%P
\qbezier(3.5,0)(3.5,3)(5,3.75)
\qbezier(5,3.75)(5.6,4)(6.2,6)

%direction
\put(2.12,3){\vector(0,1){0.1}}
\put(4.9,3){\vector(0,1){0.1}}
\put(3.5,3.7){\vector(1,0){0.1}}
\put(3.85,2.5){\vector(1,1){0.1}}

%point
\put(1.7,3.3){$\tilde{q}$}
\put(5.2,3.3){$q$}
\put(3.5,-0.3){$p$}
\put(3.5,2){$\mathcal{P}$}
\put(1.7,5){$\gamma$}
\put(5.1,5){$\mu_q$}

\end{picture}

\end{center}
\caption{The $h$ and $F$-distances from the vertex for a classical surface of revolution.}\label{fig: F_von_1}
\end{figure} 

%\bigskip

\begin{proof}
If $q=p$ the result is trivial. Let us consider $q\neq p$ to belong to a parallel $\{r=r_0\}$, that is $q$ has coordinates $(r_0,\theta_0)$, and let us consider the $h$-unit speed meridian $\mu_q$: $\{\theta=\theta_0\}$ from $p$ through $q$. Obviously $d_h(p,q)=r_0$. On the other hand, the unit speed $F$-geodesic $\mathcal{P}:[0,L]\to M$ from $p$ to $q$ can be constructed in the following way. Let us denote by $\tilde{q}$ the point on the parallel $\{r=r_0\}$ through $q$ such that $\tilde{q}=\varphi(-L,q)$ where $\varphi$ is the flow of the wind $W$, and by $\gamma$ the $h$-unit speed meridian from $p$ to $\tilde{q}$, that is $\gamma(s)=\varphi(-s,\mathcal{P}(s))$. The existence of such a point $\tilde q$ is guaranteed by the intermediate value theorem. Obviously $d_h(p,\tilde{q})=r_0=L$ since both $h$-geodesic $\gamma$ and $F$-geodesic $\mathcal{P}$ use the same unit length parameters (see Figure \ref{fig: F_von_1}).
$\qedd$
\end{proof}

\begin{remark}
We observe that the $h$-circles $\{x\in M : d_h(p,x)=\rho\}$ coincide with the $F$-circles $\{x\in M : d_F(p,x)=\rho\}$, for any $\rho>0$, i.e. the $h$-parallels coincide with the $F$-parallels.  
\end{remark}

\begin{remark}
More generally, we can define a generic abstract Finsler surface of 
revolution, not necessarily of Randers type.

A complete Finsler manifold $(M,F)$ homeomorphic to $\R^2$ that admits a point $p\in M$ such that for any two points $x$, $y\in M$, such that $d_F(p,x)=d_F(p,y)$, there exists a Finsler isometry $\varphi:M\to M$ such that $\varphi(x)=y$ is called an {\it abstract Finsler surface of 
revolution}. 

The rotational Randers metric constructed above is a special case of abstract Finsler surface of 
revolution. 

Nevertheless, it worth mentioning that the flag curvature $K$ of $F=\alpha+\beta$ is constant on each geodesic circle $\{x\in M : d_F(p,x)=\rho\}\subset M$, for any radius $\rho>0$ (see Lemma \ref{F_von_lem2}). The point $p$ is called {\it the vertex} of the surface of revolution $(M,F=\alpha+\beta)$ and in this case it coincides with the vertex of $(M,h)$. 

We restrict ourselves in the present paper to this special metric leaving the general case of an abstract Finsler surface of 
revolution for a forthcoming research.
\end{remark}

%%%%%%%%%%%%%%%%%%%%%%%%%%%%%%%%%%%%%%%%%%%%%%%%%%%%%%%%%%%%%%%%%%%%%%%%%%%%%%%%%%%%5

\subsection{The isometric embedding}\label{Isom Embed}

\quad We consider now the problem if $(M,F)$ can be isometrically embedded in a Minkowski space.

 Let us begin by constructing a rotational Minkowski metric of Randers type $\tilde{F}=\tilde{\alpha}+\tilde{\beta}$ in $\R^3$ obtained from the Zermelo navigation data $(\R^3;\delta,\widetilde{W})$, where 
$\delta=(\delta_{ij})$ is the canonical Euclidean metric of $\R^3$ and  $\widetilde{W}=(-\mu y,\mu x,0)$ is the rotation around the $z$ axis, where $\mu>0$ is a positive constant. 

First thing to notice is that
%\begin{equation*}
$|\widetilde{W}|_\delta=\mu^2(x^2+y^2)$
%\end{equation*}
and hence $|\widetilde{W}|_\delta<1$ if and only if $x^2+y^2<\frac{1}{\mu^2}$. Therefore, 
in order to obtain a positive definite Minkowski metric 
$\tilde{F}=\tilde{\alpha}+\tilde{\beta}$ we will restrict ourselves to the cylinder
\begin{equation}\label{CylinderMin}
\mathcal{U}_\mu:=\left\{(x,y,z)\in\R^3:x^2+y^2<\frac{1}{\mu^2}\right\}.
\end{equation}

We obtain immediately

\begin{proposition}
The pair $(\mathcal{U}_\mu,\tilde{F}=\tilde{\alpha}+\tilde{\beta})$ is a positive definite Minkowski space of Randers type obtained as a solution of Zermelo navigation problem in $\R^3$ with navigation data $(\delta,\widetilde{W})$, where $\mathcal{U}_\mu$ is given by \eqref{CylinderMin} where $(x,y,z,Y^1,Y^2,Y^3)$ are coordinate in $T\R^3$.
\end{proposition}
 
 Indeed, remark that $\tilde{F}$ is obtained through the Zermelo navigation process from navigation data $(\delta, \widetilde{W})$ in $\R^3$. Obviously the sectional curvature of $\delta$ is zero and $\widetilde{W}$ is Killing with respect to $\delta$, this from Theorem 3.1. in \cite{BRS} it follows that  $\tilde{F}$ must be of zero flag curvature, that is Minkowski.

%We can see that
A simple computation shows that in this case 
%\begin{lemma}\label{tilde a}
the Riemannian metric $(\tilde{a}_{ij})$ and function $(\tilde{b}_i)$ obtained through Zermelo navigation process from $\delta$  and $\widetilde{W}$ are
\begin{equation}\label{tilde a}
(\tilde{a}_{ij})=\frac{1}{\tilde{\lambda}^2}\left(
\begin{matrix}
1-\mu^2x^2 & -\mu^2xy & 0 \\
-\mu^2xy & 1-\mu^2y^2 & 0 \\
0 & 0 & \tilde{\lambda}
\end{matrix}
\right),(\tilde{b}_i)=-\frac{1}{\tilde{\lambda}}\widetilde{W},
\end{equation}
where $\tilde{\lambda}=1-\mu^2(x^2+y^2)$.
%\end{lemma}

%The proof is an elementary computation using formula \eqref{a and b}.

\begin{lemma}\label{a-space isom embedd}
\quad The mapping $\phi:M\to\R^3,(r,\theta)\mapsto(m(r)\cos \theta,m(r)\sin \theta,r)$ is an isometric embedding of $(M,a)$ in $(\R^3,\tilde{a})$, where $a=(a_{ij})$ is given in \eqref{eq 3.1} and $\tilde{a}=(\tilde{a}_{ij})$ in \eqref{tilde a}.
\end{lemma}

\begin{proof}

Taking into account that
%\begin{equation*}
%\begin{cases}
$(dx,dy,dz) =( m'\cos \theta dr - m\sin \theta d\theta ,
m'\sin \theta dr + m\cos \theta d\theta,
dr)$
%\end{cases}
%\end{equation*}
a straightforward computation shows that
\begin{equation*}
%\begin{split}
\tilde{a}=\tilde{a}_{11}(dx)^2+\tilde{a}_{22}(dy)^2+\tilde{a}_{33}(dz)^2+2\tilde{a}_{12}dxdy =a_{11}(dr)^2+a_{22}(d\theta)^2=a.
%\end{split}
\end{equation*}
$\qedd$
\end{proof}

\begin{lemma}\label{tilde b}
The linear 1-form $\beta$ is mapped to 
$\tilde{\beta}$, that is 
$\phi_*(\beta)=\tilde{\beta}$, where $\beta=b_2(r)\cdot y^2$ and $\tilde{\beta}=\tilde{b}_1Y^1+\tilde{b}_2Y^2$, $(b_i)_{i=1,2}$ is given in \eqref{eq 3.1} and $(\tilde{b}_j)_{j=1,2,3}$ in \eqref{tilde a}.
\end{lemma}

\begin{proof}

Using notations
%\begin{equation*}
%\begin{split}
$y=y^1\cdot\frac{\partial}{\partial r}+y^2\cdot\frac{\partial}{\partial \theta}=(y^i)_{i=1,2}\in TM$ and 
$Y=Y^1\cdot\frac{\partial}{\partial x}+Y^2\cdot\frac{\partial}{\partial y}+Y^3\cdot\frac{\partial}{\partial z}=(Y^j)_{j=1,2,3}\in T\R^3\simeq\R^3.
$
%\end{split}
%\end{equation*}

Then, %from Remark \ref{rem: extrinsic}
% we get
\begin{equation*}
\begin{cases}
y^1\cdot m'\cos \theta - y^2\cdot m\sin \theta = Y^1 \\
y^1\cdot m'\sin \theta + y^2\cdot m\cos \theta = Y^2
\end{cases}
\end{equation*}
and solving this linear system for $y^1,y^2$ we obtain
\begin{equation*}
\begin{cases}
y^1=\frac{1}{m'}(\cos \theta \cdot Y^1+\sin \theta \cdot Y^2) \\
y^2=\frac{1}{m}(-\sin \theta \cdot Y^1+\cos \theta \cdot Y^2).
\end{cases}
\end{equation*}

We compute now
\begin{equation*}
%\begin{split}
\phi_*(b_2y^2)=\phi_*\left(\frac{-\mu m^2}{1-\mu^2m^2}\cdot y^2\right)
%&=\phi_*\left(\frac{\mu m}{1-\mu^2m^2}\cdot \sin \theta\right)\cdot Y^1+\phi_*\left(\frac{-\mu m}{1-\mu^2m^2}\cdot \cos \theta\right)\cdot Y^2\\
=\frac{\mu y}{\tilde{\lambda}}\cdot Y^1-\frac{\mu x}{\tilde{\lambda}}\cdot Y^2=\tilde{b}_1Y^1+\tilde{b}_2Y^2.
%\end{split}
\end{equation*}
$\qedd$
\end{proof}

We obtain

\begin{proof}[Proof of Theorem \ref{embedding}]
Let us assume that there exists an Riemannian isometric embedding $\phi:(M,h)\to (\R^3,\delta)$, for instance we consider the mapping $\phi$ defined in  Lemma \ref{a-space isom embedd} (it can be easily checked that this is a Riemannian isometric embedding). From Lemmas \ref{a-space isom embedd} and \ref{tilde b} follows
that this $\phi$ 
%defined in Lemma \ref{a-space isom embedd} 
is an isometric embedding of 
the rotational Randers space $(M,F)$ into the Minkowski space $(\mathcal{U}_\mu,\tilde{F})$.

Conversely, assume  that there exists a Finslerian isometric embedding $\phi$ of 
the rotational Randers space $(M,F)$ into the Minkowski space $(\mathcal{U}_\mu,\tilde{F})$. A straightforward computation shows that the mapping $\phi$ defined in  Lemma \ref{a-space isom embedd} satisfies this requirement. Then, by same computations as above one can easily check that this $\phi$ is actually a Riemannian isometric embedding of $(M,h)$ into 
 $(\R^3,\delta)$. 
$\qedd$
\end{proof}

%A more general result about the isometric embedding of a Randers metric constructed from a 

More general results concerning isometrically embeddings for Randers type metrics with Zermelo navigation data $(h,W)$, where $h$ is an isometrically embedded Riemannian metric in $\R^3$ and $W$ is a Killing vector field, will be reported elsewhere.  

%%%%%%%%%%%%%%%%%%%%%%%%%%%%%%%%%%%%%%%%%%%%%%%%%%%%%%%%%%%%%%%%%%%%%%%%%%%%%%%%%%%%5
%%%%%%%%%%%%%%%%%%%%%%%%%%%%%%%%%%%%%%%%%%%%%%%%%%%%%%%%%%%%%%%%%%%%%%%%%%%%%%%%%%%%

\section{Geodesics of a Randers rotational surface of revolution}

\subsection{The Clairaut relation}\label{Sec3}

We are interested in finding a similar relation with the Clairaut relation for the geodesics of $(M,F)$. One can easily see that there are many directions to approach this problem. Simply study how is the $h$-Clairaut constant $\nu$ controlling the behavior of Finslerian geodesics, search for a substitute of the Clairaut constant in the Finslerian case, or can replace  $\sin\phi=\cos(\frac{\pi}{2}- \phi)$ with the Finslerian inner product $g$. 
We will consider here the simplest case. 

Remark first that $\theta$ is cyclic coordinate for the Finslerian Lagrangian $\mathcal{L}_F=F^2=(\alpha+\beta)^2$ as well, that is $\frac{\partial \mathcal{L}_F}{\partial \theta}=0$. From the general theory of calculus of variations it follows that $\frac{\partial}{\partial \theta}$ is an infinitesimal symmetry and that the Finslerian momentum
%\begin{equation*}
$p_2:=\frac{1}{2}\frac{\partial \mathcal{L}_F}{\partial y^2}$
%\end{equation*}
is a first integral for $\mathcal{L}_F$.
%We are going to compute $p_2$. 

A simple computation shows that 
%\begin{lemma}
the $\alpha$-length of the tangent vector $\dot{\mathcal{P}}$ of an $F$-geodesic $\mathcal{P} (s)$ is given by
%\begin{equation*}
%{\color{red}
$\alpha^2(\mathcal{P},\dot{\mathcal{P}})=\left(\frac{1+\mu \nu}{1-\mu^2m^2(r(s))}\right)^2.
$

 Then we get
\begin{theorem}\label{thm: Finsler mom conserv}
The conservation law for the Finslerian momentum $p_2$ is given by 
%\begin{equation*}
$p_2(s)=\frac{\nu}{1+\mu\nu}$.
%\end{equation*}
\end{theorem}
\begin{proof}
One can see that
\begin{equation*}
%\begin{split}
p_2=\frac{1}{2}\frac{\partial F^2}{\partial y^2}=F\cdot\frac{\partial F}{\partial y^2}=F\frac{\partial(\alpha+\beta)}{\partial y^2}=F\cdot\left[\frac{a_{22}y^2}{\sqrt{a_{11}(y^1)^2+a_{22}(y^2)^2}}+b_2\right],
%\end{split}
\end{equation*}
where we take into account $\frac{\partial\alpha^2}{\partial y^2}=2\alpha_{22}y^2$. 

We will evaluate now $p_2$ on the $F$-geodesic $\mathcal{P}(s)$ :
\begin{equation*}
%\begin{split}
p_2(s)
=\left[\frac{a_{22}(r(s))\dot{\mathcal{P}}^2}{\alpha(\mathcal{P},\dot{\mathcal{P}})}+b_2(r(s))\right]=\frac{\nu}{1+\mu\nu},
%\end{split}
\end{equation*}
by making use of $a_{22}(r(s))\dot{\mathcal{P}}^2=\frac{\nu+\mu m^2}{\lambda^2}$.$\qedd$
\end{proof}

We have seen that the basis of Clairaut relation for $h$-geodesic is that the inner product $h\left(\dot{\gamma},\frac{\partial}{\partial \theta}\right)=\nu$ is constant.

For the Finslerian case, we get
\begin{proposition}\label{innerPropF}
The Finslerian inner product of $\dot{\mathcal{P}}$ and $\frac{\partial}{\partial \theta}$
is constant.

\end{proposition}

\begin{proof}
 We remark first that
\begin{equation}\label{innerF}
\frac{1}{2}\frac{\partial F^2}{\partial y^2}(y^1,y^2)=g_y(y,\frac{\partial}{\partial \theta}), \text{ where } y=y^1\frac{\partial}{\partial r}+y^2\frac{\partial}{\partial \theta}\in T_{(r,\theta)}M.
\end{equation}

Indeed, by taking into account $0$-homogeneity of $g$ we have: 
\begin{equation*}
\frac{1}{2}\frac{\partial F^2}{\partial y^2}=\frac{1}{2}\frac{\partial}{\partial y^2}\left[g_{ij}(y)y^iy^j\right]=g_{21}(y)y^1+g_{22}(y)y^2.
\end{equation*}
On the other hand,
%\begin{equation*}
$g_y(y,\frac{\partial}{\partial \theta})=g_y((y^1,y^2),(0,1))=g_{12}(y)y^1+g_{21}(y)y^2
$
%\end{equation*}
and hence formula \eqref{innerF} follows.

Now, by evaluating \eqref{innerF} along $F$-geodesics $\mathcal{P}(s)$ and taking into account Theorem \ref{thm: Finsler mom conserv} we obtain 
%\begin{equation*}
$g_{\dot{\mathcal{P}}}(\dot{\mathcal{P}},\frac{\partial}{\partial \theta})=\frac{\nu}{1+\mu\nu}$. 
%\end{equation*}
$\qedd$
\end{proof}

Formally, we can define the Finslerian cosine function $\cos_F$ by
\begin{equation*}
%\begin{split}
g_y(y,W)=|y|_{g_y}\cdot|W|_{g_y}\cdot\cos_F(y,W)=\sqrt{g_y(W,W)}\cdot\cos_F(y,W).
%\end{split}
\end{equation*}

Hence, from Proposition \ref{innerPropF}, we obtain 
\begin{equation*}
%\begin{split}
g_{\dot{\mathcal{P}}}(\dot{\mathcal{P}},\frac{\partial}{\partial \theta})=\sqrt{g_{\dot{\mathcal{P}}}\left(\frac{\partial}{\partial \theta},\frac{\partial}{\partial \theta}\right)}\cdot\cos_F\left(\dot{\mathcal{P}},\frac{\partial}{\partial \theta}\right)=\sqrt{g_{22}(\mathcal{P},\dot{\mathcal{P}})}\cdot \cos_F\left(\dot{\mathcal{P}},\frac{\partial}{\partial \theta}\right)
%\end{split}
\end{equation*}
and therefore we have
\begin{corollary}
\begin{equation*}
\sqrt{g_{22}(\mathcal{P},\dot{\mathcal{P}})}\cdot \cos_F\left(\dot{\mathcal{P}},\frac{\partial}{\partial \theta}\right)=\frac{\nu}{1+\mu\nu}.
\end{equation*}
\end{corollary}

This formula is the Finslerian version of the Clairaut relation given in Theorem \ref{thm: h-Clairaut rel}.

\begin{remark}
\begin{itemize}
\item[(1)] One can now compute $g_{22}$ for the Randers metric $F=\alpha+\beta$ and substitute on the Corollary above, but we don't need to do this here.
\item[(2)] A comparison of Finslerian $\cos_F$ and usual $\cos$ should be interesting . We will leave this study for another paper. 
\end{itemize}
\end{remark}

\begin{remark}
We observe again that Clairaut relation is equivalent to saying that for the geodesics variation with the variation vector field tangent to parallel direction, the constant vector field $V=\frac{\nu}{1+\mu\nu}\cdot\frac{\partial}{\partial \theta}$ is a Jacobi vector field along the base geodesic.
\end{remark}

We denote  the angles of the $h$-geodesic $\gamma$ and the $F$-geodesic $\mathcal P$ with a meridian by $\phi$ and $\psi$, respectively.

\begin{figure}[h]%\label{fig 2}
\begin{center} \setlength{\unitlength}{1cm} 
\begin{picture}(10,4.5) 
\put(4,1){\vector(1,0){5}}
\put(5,0){\vector(0,1){5}} 
\put(4,1){\vector(1,0){3.1}} 
\put(5,1){\vector(2,1){2.8}} 
\put(5,1){\vector(1,2){0.7}}
\qbezier(7,1)(7.025,1.05)(7.05,1.1) 
\qbezier(7.1,1.2)(7.125,1.25)(7.15,1.3)
\qbezier(7.2,1.4)(7.225,1.45)(7.25,1.5)
\qbezier(7.3,1.6)(7.325,1.65)(7.35,1.7)
\qbezier(7.4,1.8)(7.425,1.85)(7.45,1.9)

\qbezier(7.5,2)(7.525,2.05)(7.55,2.1)
\qbezier(7.6,2.2)(7.625,2.25)(7.65,2.3)
\multiput(5,2.4)(0.2,0){13} {\line(1,0){0.1}} 
\qbezier(5,1.8)(5.2,1.9)(5.4,1.8) 
\qbezier(5,2)(5.6,2.3)(6.5,1.8) 
\put(5,1.4){$\phi$} 
\put(5.9,2.1){$\psi$} 
\put(4.5,4.5){$\frac{\partial}{\partial r}$} 
\put(6.5,0.5){${W}={\mu \cdot\frac{\partial}{\partial \theta}}$}
\put(9,0.4){$\frac{\partial}{\partial \theta}$} 
\put(5.7,2.7){$\dot{\gamma}$} 
\put(7.8,2.6){$\dot{\mathcal{P}}$} 
\end{picture} 
\end{center}
\caption{The angle $\psi$  between $\dot{\mathcal P}$ and a meridian.}
\end{figure}
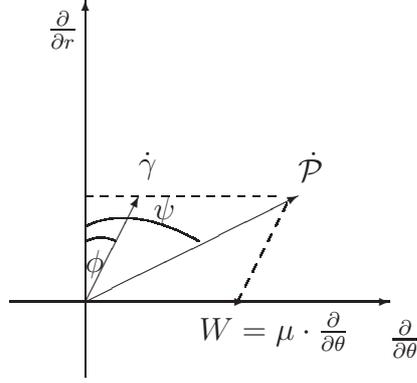 

Then by straightforward computation we obtain
\begin{equation}\label{LHS}
h(\dot\gamma,\dot{\mathcal P})=1+\mu\nu=\textrm{constant}.
\end{equation}

%Indeed, 
%$$
%h(\dot\gamma,\dot{\mathcal P})=h(\dot\gamma,\dot{\gamma}+W)=1+h(\frac{du}{ds}\frac{\partial}{\partial u}+\frac{dv}{ds}\frac{\partial}{\partial v},\mu\frac{\partial}{\partial v})=1+\mu\frac{dv}{ds}
% h(\frac{\partial}{\partial v},\frac{\partial}{\partial v})
%$$
%and using now \eqref{Riemann metric on surf of revol} and \eqref{h-prime integral}, the desired formula follows.

On the other hand, by using the definition of the scalar product, it follows %we have
\begin{equation}\label{RHS}
h(\dot\gamma,\dot{\mathcal P})=|\dot\gamma|\cdot |\dot{\mathcal P}|\cdot \cos(\psi-\phi)=
|\dot{\mathcal P}|\cdot \cos(\psi-\phi)=\sqrt{ 1+2\mu \nu+\mu^2m^2}\cos(\psi-\phi),
\end{equation}
where we remark that 
\begin{equation}\label{RHS1}
\begin{split}
|\dot{\mathcal P}(s)|& =\sqrt{h(\dot{\gamma}(s)+W,\dot{\gamma}(s)+W)}=
\sqrt{h(\dot{\gamma}(s),\dot{\gamma}(s))+2h(W,\dot\gamma)+h(W,W)}\\
& = \sqrt{ 
1+2\mu h(\frac{\partial}{\partial \theta},\dot\gamma)+\mu^2h(\frac{\partial}{\partial \theta},\frac{\partial}{\partial \theta})}  = 
\sqrt{ 
1+2\mu \frac{d\theta}{ds}|\frac{\partial}{\partial \theta}|^2+\mu^2|\frac{\partial}{\partial \theta}|^2
}\\
& = \sqrt{ 1+2\mu \frac{d\theta}{ds} m^2+\mu^2m^2}= \sqrt{ 1+2\mu \nu+\mu^2m^2}.
\end{split}
\end{equation}

\begin{proof}[Proof of Theorem \ref{Finslerian Clairaut relation}]
1. It follows immediately from 
relations \eqref{LHS} and \eqref{RHS}.

%From we obtain
%\begin{theorem}[Finslerian Clairaut relation]
%If $\gamma(s)=(u(s),v(s))$ is an $h$-geodesic and $\mathcal P(s)$ the %corresponding $F$-geodesic  on a surface of revolution, then 
%the angle $\psi$ between $\dot{\mathcal P}(s)$ and the profile curve passing through $\mathcal P(s)$ satisfies 
%we have
%\begin{equation}\label{Clairaut Fins}
%\sqrt{ 1+2\mu \nu+\mu^2f^2}\cos(\psi-\theta)=1+\mu\nu.
%\end{equation}
%\end{theorem}

%One can see that this is a natural generalization of the Riemannian Clairaut %relation.

%\begin{remark}[Another version of Finslerian Clairaut relation]

2. 
Another version of Finslerian Clairaut relation is also possible. 

We compute as before 
\begin{equation}\label{second form 1}
%\begin{split}
h(\dot{\mathcal P}(s),\frac{\partial}{\partial \theta})=h(\dot{\gamma}(s)+W,\frac{\partial}{\partial \theta})=
h(\dot{\gamma}(s),\frac{\partial}{\partial \theta})+h(W(s),\frac{\partial}{\partial \theta})
%=\Bigl(\frac{d\theta}{ds}+\mu\Bigr)\cdot
%|\frac{\partial}{\partial \theta}|^2
=\Bigl(\frac{d\theta}{ds}+\mu\Bigr)m^2
.
%\end{split}
\end{equation}

On the other hand, from the inner product definition we have
\begin{equation}\label{inner_prod 1}
h(\dot{\mathcal P}(s),\frac{\partial}{\partial \theta})=|\dot{\mathcal P}(s)|\cdot |\frac{\partial}{\partial \theta}|\cdot \cos(\frac{\pi}{2}-\psi),
\end{equation}
where $|\cdot|=\sqrt{h(\cdot,\cdot)}$.

Using now  \eqref{second form 1} and \eqref{RHS1}, relation  \eqref{inner_prod 1} implies the relation. 

%\begin{theorem}
%If $\gamma(s)=(u(s),v(s))$ is a $h$-geodesic and $\mathcal P(s)$ the corresponding $F$-geodesic  on a surface of revolution, then the angle $\psi$ between $\dot{\mathcal P}(s)$ and the profile curve passing through $\mathcal P(s)$ satisfies 
%\begin{equation}\label{Clairaut Fins 1}
%f\sin\psi=\frac{\nu+\mu f^2}{\sqrt{ 1+2\mu \nu+\mu^2f^2}}.
%\end{equation}
%\end{theorem}
$\qedd$
\end{proof}

%Obviously, the two versions of the Finslerian Clairaut relation are equivalent.
%\end{remark}

\subsection{Geodesics behaviour on a Randers surface of revolution}\label{sec: geodesics behaviour}

We are  going to characterise the behaviour of the Randers geodesics by making use of the Riemannian Clairaut relation for $h$ or/and one of the Finslerian versions.

Let $(M,F)$ be a forward complete non-compact Finsler surface. A point $p\in M$ is called {\it a pole} if any two geodesics from $p$ do not meet again. In other words, the cut locus of $p$ is empty. 

A unit speed geodesic of $(M,F)$ is called {\it a forward ray} if $d_F(\gamma(0),\gamma(s))=s$, for all $s\geq 0$. In other words  a forward ray is a globally forward minimizing $F$-geodesic.

\begin{proposition}\label{twisted rays}
If $\gamma(s)=(r(s),\theta(s))$ is an $h$-ray, then the twisted ray $\mathcal P(s):=(r(s),\theta(s)+\mu s)$ is a forward ray. 
\end{proposition}
\begin{proof}
Since $\gamma$ is $h$-ray it follows $\mathcal P(s)$ is $F$-unit speed geodesic and taking into account that 
$h(\dot{\gamma}(s),\dot{\gamma}(s))= 
F(\dot{\mathcal P}(s),\dot{\mathcal P}(s))=1$ it follows $\mathcal P(s)$ is $F$ forward ray. $\qedd$
\end{proof}

%We obtain immediately

It follows

\begin{proposition}
\begin{enumerate}
\item If $\gamma(s)=(r(s),\theta_0)$ is a meridian, then the twisted meridian  $\mathcal P(s)=(r(s),\theta_0+\mu s)$ is a forward ray.
\item A twisted meridian can not be tangent to a parallel nor to a meridian.
\item The twisted meridians are not h-geodesics. %, provided $f$ is not constant.
\end{enumerate}
\end{proposition}
\begin{proof}

1. It follows immediately from Proposition \ref{twisted rays}.

2. Since $\mathcal P(s)$ is a twisted meridian, the corresponding $h$-geodesic $\gamma$ must be a meridian, that is, $\phi=0$ and $\nu=0$ along $\gamma$.

Then the Clairaut relations \eqref{Clairaut Fins} and \eqref{Clairaut Fins 1} for our Finsler metric read
\begin{equation}\label{Clair for TM}
\cos\psi=\frac{1}{\sqrt{1+\mu^2m^2}},
\end{equation} 
and 
\begin{equation}\label{Clair 1 for TM}
\sin\psi=\frac{\mu m}{\sqrt{1+\mu^2m^2}},
\end{equation} 
respectively.

If the twisted meridian $\mathcal P(s)$ is tangent to a parallel % that is not $h$-geodesic 
in a point $(r(s_1),\theta(s_1))$ it means $\psi(s_1)=\frac{\pi}{2}$, and Finslerian Clairaut relation \eqref{Clair for TM} gives
$
\cos(\frac{\pi}{2})=0=\frac{1}{\sqrt{1+\mu^2m^2(r(s_1))}}
$
that is not possible. 

Likely, if  $\mathcal P(s)$ is tangent to a meridian in $(r(s_1),\theta(s_1))$ it means 
$\psi(s_1)=0$ and Finslerian Clairaut relation \eqref{Clair 1 for TM} gives 
$
\sin 0=0=\frac{\mu m}{\sqrt{1+\mu^2m^2}},
$
that is not possible either.

3. Let us assume that $\gamma(s)=(r(s),\theta_0)$ is a meridian on $M$, that is, $\gamma$ is a $h$-geodesic with Clairaut constant 
$\nu=0$. If the twisted meridian $\mathcal P(s)=(r(s),\theta_0+\mu s)$ would also be an $h$-geodesic, then it should satisfy the Riemannian Clairaut relation $m(r(s))\sin \psi(s)=$constant. 

However, Finslerian Clairaut relation for the twisted meridian  $\mathcal P(s)$ given in \eqref{Clair 1 for TM} implies
$$
m(r(s))\sin \psi(s)=\frac{\mu m^2}{\sqrt{1+\mu^2m^2}},
$$
and this cannot be constant except for $m=$ constant, but this is not possible due to our definition of $M$.

$\qedd$
\end{proof}

\begin{remark}
Relations \eqref{Clair for TM} and \eqref{Clair 1 for TM}
give the following Finslerian Clairaut relation for twisted meridians 
\begin{equation}
m(r(s))|\cot\psi(s)|=\frac{1}{\mu}.
\end{equation}
\end{remark}

If $\gamma:\{r=r_0\}$ is a parallel on $M$ such that $m'(r_0)=0$, then $\mathcal P(s)=(r_0,\theta(s)+\mu s)$ is the same parallel $\gamma$ as set of points (as non-parametrized curve). We get
\begin{proposition}
Parallels $\mathcal P(s)=(r_0,\theta(s)+\mu s)$ on $M$, such that  $m'(r_0)=0$, are geodesics of $(M,F)$.  
\end{proposition}

We also have
\begin{proposition}
Meridians can not be F-geodesics. %, provided $f$ is not constant.
\end{proposition}
\begin{proof}
Assume that the $F$-geodesic $\mathcal P$ is a meridian, that is we can write $\mathcal P(s)=(r(s),\theta_0)$, and taking into account that this is also an $F$-geodesic it follows that it must exist an $h$-geodesic $\gamma(s)=(r(s),\tilde \theta(s))$ such that  $\mathcal P(s)=(r(s),\theta_0)
=(r(s),\tilde \theta(s)+\mu s)$. This means that the pre-image $h$-geodesic is 
$\gamma(s)=(r(s),\tilde \theta(s)=\theta_0-\mu s)$, and thus $\frac{d\tilde \theta(s)}{ds}=-\mu$.

 But $\gamma(s)$ being an unit speed $h$-geodesic means
 %{\color{red}
 \begin{equation}\label{eq h_1}
\left(\frac{dr}{ds}\right)^2 =1-\mu^2m^2(r(s))
\end{equation}%}
and
\begin{align}
& \frac{d^2r}{ds^2}-mm'\mu^2=0\\% \label{eq 3}\\
& 2\mu \frac{m'}{m}\frac{dr}{ds}=0.% \label{eq 4}.
\end{align}

Since $\frac{dr}{ds}$ cannot vanish due to \eqref{eq h_1} and positive definiteness of $F$, the second equation above shows that this is possible only in the case $m'(r(s))=0$, that is, $m$ is constant along a meridian, but this is not possible. $\qedd$ 
\end{proof}

We will find the explicit equation of a segment of a geodesic of $(M,F)$, i.e. $\mathcal{P}^2(\mathcal{P}^1)$.

We recall that for the unit speed $h$-geodesic $\gamma(s)=(r(s),\theta(s))$ we have
%\begin{equation*}
$\left(\frac{dr}{ds}\right)^2=\frac{m^2-\nu^2}{m^2}$.
%\end{equation*}

It results
%\begin{equation*}
$\frac{ds}{dr}=m \cdot\sqrt{\frac{1}{m^2-\nu^2}}$
%\end{equation*}
and therefore from \eqref{eq 1} we have
%\begin{equation*}
$\frac{d\theta}{dr}=\frac{\nu}{m}\cdot\sqrt{\frac{1}{m^2-\nu^2}}$.
%\end{equation*}
Using these, we can write
\begin{equation*}
\begin{split}
\frac{d\mathcal{P}^2}{d \mathcal{P}^1} &= \frac{d\theta}{dr}+\mu\cdot\frac{ds}{dr} \\
&= \frac{\nu}{m}\cdot\sqrt{\frac{1}{m^2-\nu^2}}+\mu\cdot m \cdot\sqrt{\frac{1}{m^2-\nu^2}} \\
& = \left(\frac{\nu}{m}+\mu\cdot m\right)\sqrt{\frac{1}{m^2-\nu^2}},
\end{split}
\end{equation*}
hence, we get
\begin{equation}\label{P2(u)}
\begin{split}
\mathcal{P}^1=r,\ 
\mathcal{P}^2&=\int \left(\frac{\nu}{m}+\mu\cdot m\right)\sqrt{\frac{1}{m^2-\nu^2}} dr+C_1\\
&=\theta(r)+\mu\int \cdot\frac{m}{\sqrt{m^2-\nu^2}} dr+C_1,
\end{split}
\end{equation}
where $C_1$ is the integration constant.

If we denote
\begin{equation}\label{xi,eta}
\xi(r,\nu):=\frac{\nu}{m}\sqrt{\frac{1}{m^2-\nu^2}} \text{, } 
\eta(r,\nu):=\frac{m}{\sqrt{m^2-\nu^2}}
\end{equation}
for $m(r)>|\nu|$, then we get
\begin{proposition}
Let %$\mathcal{P}(s)=(u(s),v(s)+\mu s)$ 
$\gamma:[a,b)\to M$, $\gamma(s)=(r(s),\theta(s))$ be a unit-speed Riemannian $h$-geodesic whose Clairaut's constant $\nu$ is nonzero. If $r'(s)$ is nonzero on $[a,b)$ then the geodesic $\mathcal{P}$ parametrized by $u$ satisfies
\begin{equation}\label{m}
\mathcal{P}^2(b)-\mathcal{P}^2(a) \equiv \epsilon 
\int_{r(a)}^{r(b)}\left(
\xi(r,\nu)+\mu\eta(r,\nu)\right)dr\quad \mod 2\pi,
\end{equation}
\begin{equation}\label{n}
b-a=\epsilon\int_{r(a)}^{r(b)}\eta(r,\nu)dr,
\end{equation}
where $\xi$ and $\eta$ are the functions defined in \eqref{xi,eta}, and $\epsilon$ denoted the sign of $r'(s),s\in[a,b)$.
\end{proposition}
\begin{proof}
It easy to see that \eqref{m} follows from \eqref{P2(u)} and in the fact that
%\begin{equation*}
$b-a=\int_a^bds=\int_{r(a)}^{r(b)}\frac{dr}{r'(s)}=\epsilon\int_{r(a)}^{r(b)}\eta(r,\nu)dr$.
$\qedd$
%\end{equation*} 
\end{proof}

Remark that by combining \eqref{m} and \eqref{n} we get

\begin{equation*}
\mathcal{P}^2(b)-\mathcal{P}^2(a)=\theta(b)-\theta(a)+\mu(b-a),
\end{equation*}
a formula is accord with Theorem \ref{thm:global Finsler geodesics}. 

Similar with the Riemannian case we have
\begin{proposition}
Let $\mathcal{P}:I\to M$ be a Finslerian unit speed geodesic. If $\mathcal{P}=(r(s),\theta(s)+\mu s)$ is not a parallel then the zero points of $r'$ are discrete. Furthermore, if $r'=0$ for some $s_0\in I$ then $m'(r(s_0))$ is nonzero.
\end{proposition}

\begin{proof}
Let $\mathcal{P}=(r(s),\theta(s)+\mu s)$ be a Finslerian unit speed geodesic that is not a parallel. 

\begin{itemize}
\item If $\mathcal{P}$ is a meridian. Then conclusion is obvious.
\item If $\mathcal{P}$ is not a meridian, i.e. $\mathcal{P}$ do not pass through the vertex of $M$ and $r'(s_0)=0$, then $\mathcal{P}$ is tangent to the parallel $r=r(s_0)$ but $\mathcal{P}$ is not a parallel, and therefore $m'(r(s_0))\neq 0$. Since $\mathcal{P}^1(s)=r(s)$ from the equations of the $h$-geodesics it follows 
%\begin{equation*}
 $r''(s_0)\neq 0.$
%\end{equation*} 
\end{itemize}

That is, $s_0$ is a critical non-degenerate point for the function $r$ 
% Since $s_0$ is arbitrary it follows $r$ is a Morse function 
and therefore its critical points are discrete.$\qedd$
\end{proof}

Another interesting property of geodesics on a surface of revolution is the following:

\begin{proposition}\label{prop asym}
A geodesic $\mathcal P$ of $(M,F)$ can not be asymptotic to a parallel which is not geodesic.
\end{proposition}

\begin{proof}
Recall that the same property holds for Riemannian geodesics $\gamma$ of the surface of revolution $(M,h)$ (see for example \cite{AT}). 

We assume that the $F$-geodesic $\mathcal P$ is asymptotic to a parallel $\{r=r_0\}$ which is not a geodesic, that is $m'(r_0)\neq 0$. This means that $\{r=r_0\}$ is not geodesic for the Riemannian metric $h$, nor for the Randers metric $F$. Since $\mathcal P$ is an $F$-geodesic it follows that it exists a unit speed $h$-geodesic $\gamma(s)=(r(s),\theta(s))$ such that 
$\mathcal P(s)=(r(s),\theta(s)+\mu s)$. 

On the other hand, this formula shows that $\mathcal P$ asymptotic to $\{r=r_0\}$ means that $\gamma(s)$ must be asymptotic to $\{r=r_0\}$. But this is not possible because the Riemannian geodesic $\gamma(s)$ can not be asymptotic to a parallel which is not a geodesic.$\qedd$
\end{proof}

%\begin{lemma}
%If $\gamma:\{r=r_0\}$ is a parallel on $M$, then the Finslerial length of $\gamma$ is
%\begin{equation}\label{F length of parallels}
%\mathcal L_F(u_0)=\frac{2\pi m(r_0)}{1+\mu m(r_0)}.
%\end{equation}
%\end{lemma}

We have shown that the parallels and meridians can be geodesics for $F$ and $h$ in the same time. What about the rest of the geodesics? In particular we would like to know if $F$ is a Riemannian projectively equivalent surface. We will show that this is not the case.

Straightforward computations show
\begin{proposition}
\begin{enumerate}
\item The Riemannian metrics $a$ and $h$ are not projectively equivalent. %, provided $f$ is not constant.
\item The Riemannian metric $a$ and the Randers metric $(M,F)$ are not projectively equivalent. %, provided $f$ is not constant.
\item The parallels and meridians of $M$ are geodesics for $(M,a)$.
\end{enumerate}
\end{proposition}

In other words, an $h$-geodesic that is not a parallel nor a meridian is not a geodesic of the Randers metric $F$. This shows that actually the geodesics of these two structures are different.  Obviously the twisted meridians are $F$-geodesics, but they can not be $h$-geodesics, provided $m(r)$ is not constant, that is not possible in the present case. % except the cylinder case. %Indeed, for a meridian $\gamma(s)=(u(s),v_0)$, the twisted meridian is $\mathcal P(s)=(u(s),v_0+\mu s)$ and this is an $h$-geodesic if and only if it satisfies equations \eqref{eq 3} and \eqref{eq 4}, but this obviously implies $f'(u)=0$. 

%Let us study the geodesics of the Riemannian structure $a$ on $M$. Indeed, we have
%\begin{proposition}
%The parallels and meridians of $M$ are geodesics for $(M,a)$. 
%\end{proposition}

\begin{example}[A Randers paraboloid-like surface of revolution]
We start by constructing a rotational Randers metric on the surface of revolution with profile curve 
\begin{equation}
m:[0,\infty)\to \R,\qquad m(r)=\frac{r}{\sqrt{\mu^2r^2+1}} %\frac{1}{\mu}\frac{r}{\sqrt{{r^2+1}}},
\end{equation}
where $\mu$ is a positive constant. This function is bounded $m(r)<\frac{1}{\mu}$ and when revolved around $z$ axis it gives a smooth surface of revolution, homeomorphic to $\R^2$, that we call {\it paraboloid-like}. 

%\end{enumerate}
%\end{lemma}

\begin{figure}[h]
\begin{center}\setlength{\unitlength}{1cm} 
\begin{subfigure}[b]{0.3\textwidth}
\includegraphics[width=\textwidth]{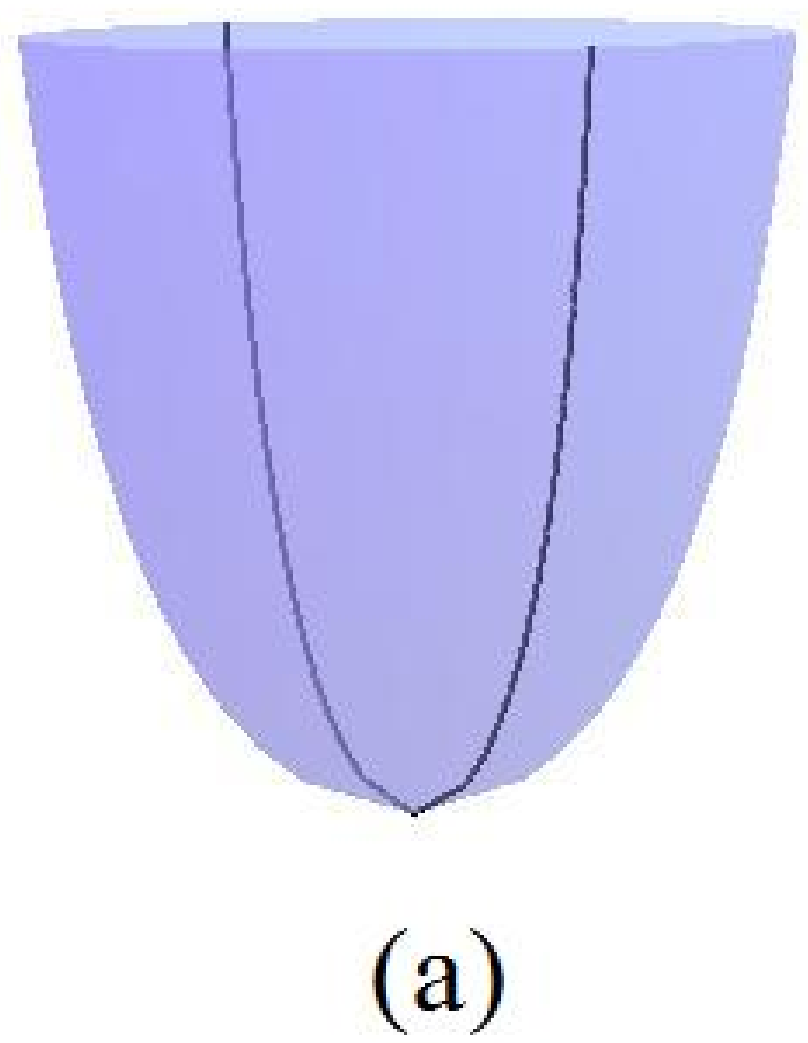}
\end{subfigure}
\begin{subfigure}[b]{0.3\textwidth}
\includegraphics[width=\textwidth]{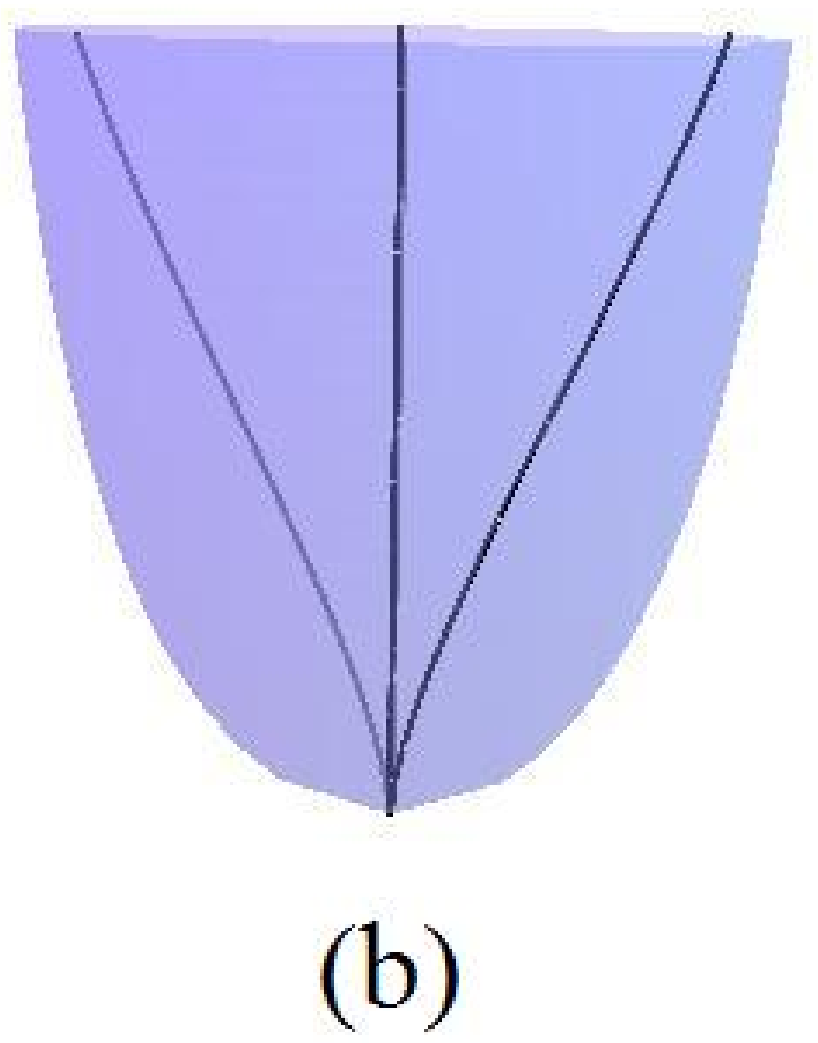}
\end{subfigure}
\begin{subfigure}[b]{0.3\textwidth}
\includegraphics[width=\textwidth]{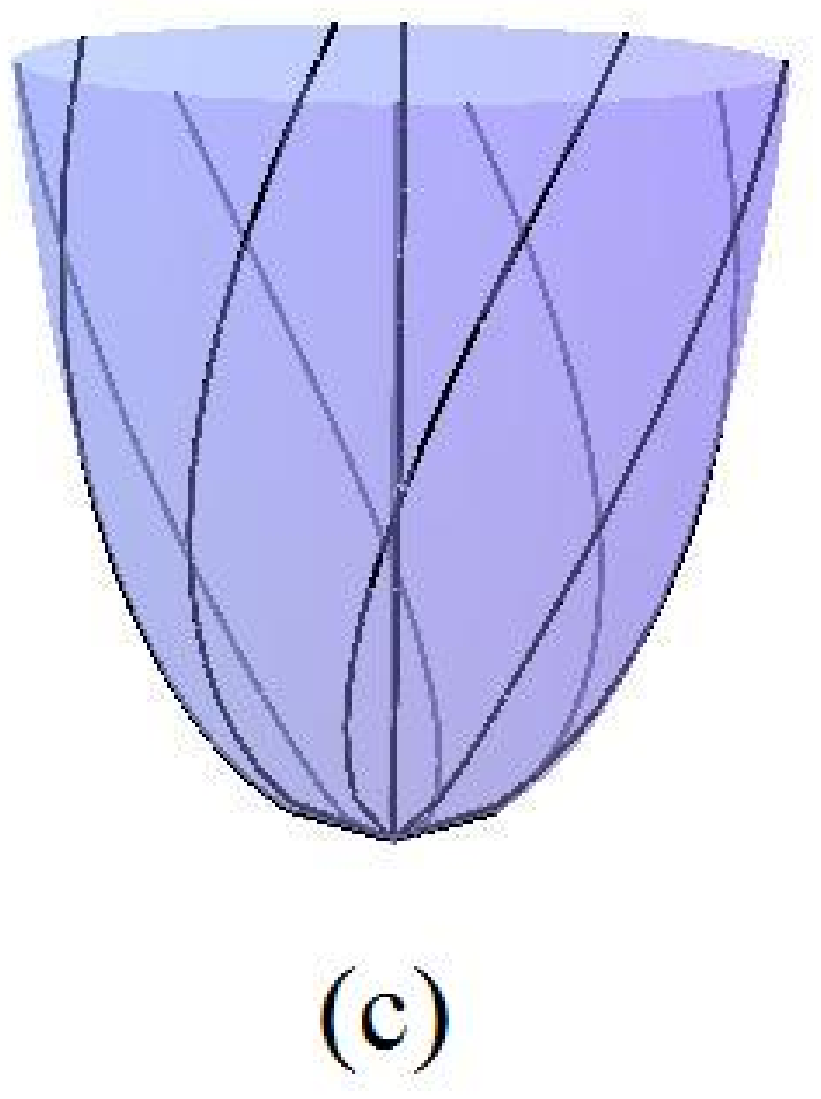}
\end{subfigure}
\begin{subfigure}[b]{0.3\textwidth}
\includegraphics[width=\textwidth]{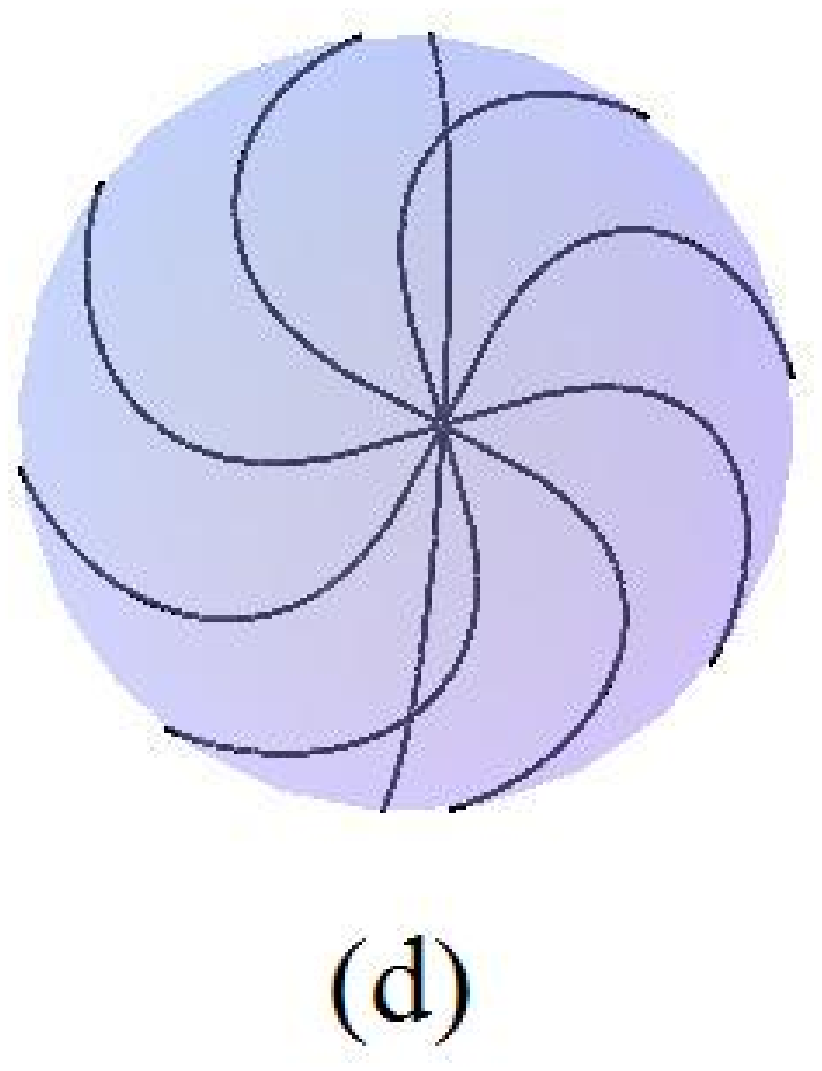}
\end{subfigure}

\caption{A Randers paraboloid-like surface of revolution for $\mu=1$. The paraboloid-like with a meridian (a); the  paraboloid-like with a meridian (the straight line in the middle) and the same meridian twisted by a wind with $\mu=1$ (b); the paraboloid-like seen from the side with one meridian (the straight line in the middle) and four twisted meridians at $\frac{\pi}{4}$ from each other (c); same picture seen from the North pole (d).}
\end{center}
\end{figure}

If we consider the Riemannian surface of revolution $(M,h)$, then  
from general theory one can easily see that meridians are $h$-geodesics and there are no 
parallel geodesics on $M$. 
%\begin{lemma}
%\begin{enumerate}
An $h$-geodesic of $(M,h)$ that is not a meridian, when traced in the direction of increasing parallels radii, intersect infinitely many times all the meridians. Moreover, an $h$-geodesic of $(M,h)$ that is not a meridian, intersects itself an infinite number of times. The proofs are similar to the general case (see for example \cite{AT}). 

%We consider now $M$ as a Finsler surface with the rotational Randers metric defined above. 

\begin{proposition}
Let $(M,F)$ be a Randers paraboloid-like surface of revolution. 
\begin{enumerate}
\item There is no parallel geodesic.
\item The twisted meridians are $F$-geodesics that intersect infinitely many times all meridians of $M$. 
\item A geodesic that is not a twisted meridian intersects itself an infinite number of times.
\end{enumerate}
\end{proposition}
\begin{proof}
The first and second statements are obvious from the previous discussions.

The third statement follows from the fact that an $h$-geodesic $\gamma$ of $M$ that is not a meridian intersects itself an infinite number of times.
$\qedd$
\end{proof}
\end{example}

%%%%%%%%%%%%%%%%%%%%%%%%%%%%%%%%%%%%%%%%%%%%%%%%%%%%%%%%%%%%%%%%%%%%%%%%%%%%%

%\subsection{Finsler geodesics versus Riemannian geodesics}

%We have shown many similarities in the local and global behaviour of $F$ and $h$-geodesics on the surface of revolution $M$. Now we In the present section we will exploit some differences between the behaviour of these curves.

%%%%%%%%%%%%%%%%%%%%%%%%%%%%%%%%%%%%%%%%%%%%%%%%%%%%%%%%%%%%%%%%%%%%%%%%%%%%%%

\section{Rays, poles and cut locus of a Randers rotational surface of revolution}

\subsection{Rays and poles}

We will consider in the following a rotational Randers surface of revolution $(M,F)$ which is forward complete, non-compact and homeomorphic to $\R^2$. Let $p$ be the vertex of $M$. 

%\begin{remark}
%The $h$-length of the same paralled is simply 
%$\mathcal L_h(u_0)={2\pi f(u_0)}$.
%\end{remark}
\begin{proposition}\label{prop: parallel length lim}
If $\liminf_{r\to 0}\mathcal L_F(r)=0$ then for any point $q\neq p$, the sub-ray 
$\mathcal P_q|_{[d(p,q),\infty)}$ of the twisted meridian $\mathcal P_q$ from $p$ through $q$ is the unique $F$ forward ray emanating from $q$.
\end{proposition}
 \begin{proof}
 
 First of all, taking into account that the $h$-length of the parallel is $\mathcal L_h(r)=2\pi m(r)$, 
 by comparing with Corollary \ref{thm: 2 parallels length}
 %\eqref{F length of parallels} 
 we observe that $\liminf_{r\to 0}\mathcal L_F(r)=0$ is equivalent to $\liminf_{r\to 0}\mathcal L_h(r)=0$, and therefore on $(M,h)$ the only $h$-ray from $q$ is the sub-ray of the meridian from $p$ through $q$. It follows that the sub-ray $\mathcal P_q|_{[d(p,q),\infty)}$ of the twisted meridian $\mathcal P_q$ from $p$ through $q$ is a forward ray of $(M,F)$ emanating from $q$.
 
 We show that this is the unique such ray. 
 Assume $\gamma$ is an $F$ forward ray which is not tangent to any twisted meridian, that is $\nu\neq 0$. Then the hypothesis and Clairaut relation \eqref{h-prime integral} implies $\gamma$ must be bounded and therefore it cannot be forward ray.$\qedd$ 
 
  \end{proof}

\begin{proof}[Proof of Theorem \ref{thm:poles}]
Since our profile function $m$ is bounded, i.e. $m(r)<\frac{1}{\mu}$, it follows 
$
\frac{1}{\mathcal L_h^2(r)}=\frac{1}{4\pi^2}\frac{1}{m^2(r)}\geq 
\frac{\mu^2}{4\pi^2}
$
and hence
$$
\int_1^\infty \frac{1}{\mathcal L_h^2(r)}dr=\lim_{\tau\to \infty}\int_1^\tau \frac{1}{4\pi^2}\frac{1}{m^2(r)}dr \geq 
\frac{\mu^2}{4\pi^2}\lim_{\tau\to \infty}\int_1^\tau dr=
\frac{\mu^2}{4\pi^2}\lim_{\tau\to \infty}(\tau-1)=\infty.
$$

Therefore we obtain $\int_1^\infty \frac{1}{\mathcal L_h^2(r)}dr=\infty$ and Lemma \ref{paralells h-length} implies that for the Riemannian surface of revolution $(M,h)$ the vertex $p$ is the unique pole. The conclusion follows from Propositions \ref{twisted rays} and \ref{prop: parallel length lim}.$\qedd$
\end{proof}

\begin{remark}
In this case, the Busemann function $\bold{b}_\gamma$ of a ray $\gamma$ in $(M,F)$ coincides with the distance from $p$ up to a constant, i.e. $\bold{b}_\gamma(x)=d_F(p,x)+$constant, for $x\in M$, the level sets $\bold{b}_\gamma^{-1}$ are parallels on $M$, and $\bold{b}_\gamma$ is an exhaustion (see \cite{Oh}, \cite{Sa} for details on Busemann functions for Finsler manifolds).
\end{remark}

%\begin{remark}
%Many other local and global properties of $F$-geodesics can be further obtained from the corresponding properties of the $h$-geodesics via formula \eqref{global Finsler geodesics}. Indeed, Proposition 6 and Corollary 8 in \cite{R} shows that the conjugate points and the cut points of a point $p:=\mathcal P(0)$ can be obtained in the same way from the conjugate points and cut points of the corresponding Riemannian $h$-geodesic, respectively. We will discuss more details about these elsewhere.   
%\end{remark}
% % % % % % % % % % % % % % % % % % % % % % % % % % % % % %
% % % % % % % % % % % % % % % % % % % % % % % % % % % % % % % %
\subsection{von Mangoldt surfaces}\label{sec: von Mangoldt}

\quad Recall that in the Riemannian case von Mangoldt surfaces are surfaces of revolution with nice properties. We are going to introduce here some Finslerian equivalent of these.

\begin{lemma}\label{F_von_lem2}
The flag curvature $\mathcal{K}$ of the Randers rotational metric $(M,F=\alpha+\beta)$ given by \eqref{eq 3.1} lives on the base manifold $M$. Moreover $\mathcal{K}=G$, where $G$ is the Gauss curvature of $(M,h)$.
\end{lemma}

\begin{proof}
Firstly we recall that any Riemannian surface $(M,h)$ is an Einstein manifold with Ricci scalar $Ric^{(h)}=G(x)$. Two dimensional Einstein spaces are therefore not interesting for Riemannian geometry, but this is not the case for Finslerian case.

\quad Let us recall a result from \cite{BR}. Consider a Randers manifold $(M,F=\alpha+\beta)$ solution of the Zermelo's navigation problem with navigation data $(h,W)$, where $(M,h)$ is a non-flat Riemannian manifold. Then $(M,F)$ is Finsler-Einstein with Ricci scalar $Ric^{(F)}=\mathcal{K}(x)$ if and only if
 $(M,h)$ is Einstein with Ricci scalar $Ric^{(h)}=\mathcal{K}(x)$, and 
$W$ is Killing vector field for $(M,h)$.

\quad Let us particular this result to the case of the Randers rotational surface described in the present paper. Based on what we observed already it follows that on $(M,F=\alpha+\beta)$ is always Finslerian-Einstein with Ricci scalar $Ric^{(F)}=\mathcal{K}(x)$, where $\mathcal{K}$ is the sectional curvature of $(M,F)$. Indeed, in the 2-dimensional case, if we consider an $g$-orthonormal basis $\{e_1,e_2\}$ of $T_xM$, then

\begin{equation*}
\mathcal{K}=R^{\ 1}_{2\ 12}=Ric^{(F)},
\end{equation*} 
where $g$ is the Hessian of $F^2$, and $R$ the Riemannian curvature tensor of $F$ the Finsler metric (see for example \cite{BCS}, p.99).
$\qedd$
\end{proof}

\quad We give the following general definition.

\begin{definition}
The Finsler surface of revolution $(M,F)$ is called a {\it Finsler von Mangoldt} surface if, for any two points $x_1,x_2\in M$ such that
\begin{equation*}
d_F(p,x_1)\geq d_F(p,x_2)
\end{equation*}
we have
\begin{equation*}
\mathcal{K}(x_1,y_1)\leq \mathcal{K}(x_2,y_2) \quad\text{ for all } y_1\in \widetilde{T_{x_1}M},\quad y_2\in \widetilde{T_{x_2}M},
\end{equation*} 
where $\widetilde{T_{x_1}M}=T_{x_1}M\setminus\{0\}$, $\widetilde{T_{x_2}M}=T_{x_2} M\setminus\{0\}$.
\end{definition}

Obviously this is the natural generalisation of the Riemannian von Mangoldt surfaces to the Finslerian setting.

\begin{proposition}
The Randers rotational surface of revolution $(M,F=\alpha+\beta)$ is a Finsler von Mangoldt surface if and only if $(M,h)$ is a Riemannian von Mangoldt surface.
\end{proposition}

\begin{proof}
Assume $(M,h)$ is von Mangoldt, that is $G(x)\leq G(y)$ for any points $x,y\in M$ such that $d_h(p,x)\geq d_h(p,y)$. Lemmas \ref{F_von_lem1} and \ref{F_von_lem2} imply $(M,F)$ is Finsler von Mangoldt.

\quad Conversely, if $(M,F)$ is Finsler von Mangoldt, then $(M,h)$ must be von Mangoldt.
$\qedd$
\end{proof}

Now we can easily characterise the cut locus of our Randers rotational surface.

\begin{remark}\label{rem: constructing F-geodesics}
\begin{enumerate}

\item Recall that an $F$-geodesic ray from $p$ is obtained by twisting a meridian on $M$. 

More precisely, as explained already in the proof of Lemma \ref{F_von_lem1} we can construct the $F$-ray  from $p$ through any point $q\neq p$ as follows:
\begin{enumerate}
\item Take the parallel $\gamma:\{r=r(q)\}$ through $q$.
\item Consider a point $q^-$ on this parallel such that $\varphi(\rho,q^-)=q$, where $\rho:=d_h(p,q)$. Obviously such a point always exists on the universal covering $\tilde{\gamma}:[0,\infty)\to M$ of the parallel $\gamma$ by the intermediate value theorem.
\item Consider the meridian $\mu_{q^-}$ from $p$ through $q^-$. 
\end{enumerate}

Then the $F$-geodesic $\mathcal{P}_q:[0,\infty)\to M$, $\mathcal{P}_q(s)=\varphi(s,\mu_{q^-}(s))$ from $p=\mathcal{P}_q(0)=\mu_{q^-}(0)$ through $q$ is obtained by twisting the meridian $\mu_{q^-}$ as shown by Theorem \ref{thm:global Finsler geodesics} (see Figure \ref{F-geod to q}).

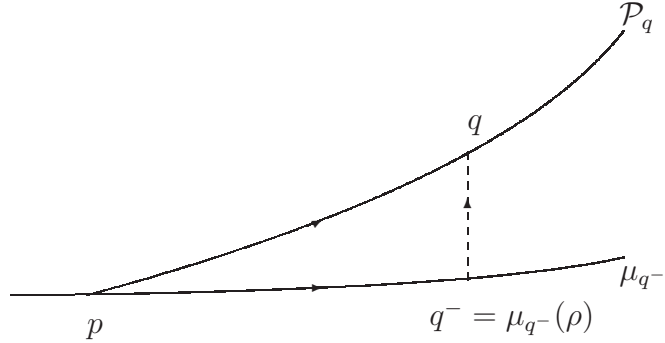
\begin{figure}[h]%\label{fig 2}
\begin{center} \setlength{\unitlength}{1cm} 
\begin{picture}(10,4.5) 
\qbezier(1,1)(7.025,1.05)(9.05,1.5)
\qbezier(2,1)(7.5,2.5)(9.05,4.5)
\put(2,0.5){$p$}
\put(6.5,0.6){$q^-=\mu_{q^-}(\rho)$}
%\multiput(7,0.6)(7,2.6){13} {\line(0,1){0.1}} 
\multiput(7,1.2)(0,0.2){9} {\line(0,1){0.1}}
\put(7,3.2){$q$}
\put(9,1.2){$\mu_{q^-}$}
\put(9,4.6){$\mathcal{P}_{q}$}
\put(5,1.1){\vector(1,0){0.1}}
\put(5,1.95){\vector(2,1){0.1}}
\put(7,2.2){\vector(0,1){0.1}}
%\put(1.1){$\circ(0,0){0.1}$}
%-------------------------------
%\put(4,1){\vector(1,0){5}}
%\put(5,0){\vector(0,1){5}} 
%\put(4,1){\vector(1,0){3.1}} 
%\put(5,1){\vector(2,1){2.8}} 
%\put(5,1){\vector(1,2){0.7}}
%\qbezier(7,1)(7.025,1.05)(7.05,1.1) 
%\qbezier(7.1,1.2)(7.125,1.25)(7.15,1.3)
%\qbezier(7.2,1.4)(7.225,1.45)(7.25,1.5)
%\qbezier(7.3,1.6)(7.325,1.65)(7.35,1.7)
%qbezier(7.4,1.8)(7.425,1.85)(7.45,1.9)

%\qbezier(7.5,2)(7.525,2.05)(7.55,2.1)
%\qbezier(7.6,2.2)(7.625,2.25)(7.65,2.3)
%\multiput(5,2.4)(0.2,0){13} {\line(1,0){0.1}} 
%qbezier(5,1.8)(5.2,1.9)(5.4,1.8) 
%\qbezier(5,2)(5.6,2.3)(6.5,1.8) 
%\put(5.02,1.3){$\phi$} 
%\put(5.9,2.1){$\psi$} 
%\put(4.5,4.5){$\frac{\partial}{\partial r}$} 
%\put(6.5,0.5){${W}={\mu \cdot\frac{\partial}{\partial \theta}}$}
%\put(9,0.4){$\frac{\partial}{\partial \theta}$} 
%\put(5.7,2.7){$\dot{\gamma}$} 
%\put(7.8,2.6){$\dot{\mathcal{P}}$} 
\end{picture} 
\end{center}
\caption{The $F$-geodesic from $p$ through $q$.}\label{F-geod to q}
\end{figure} 

%Then by straightforward computation we obtain
%\begin{equation}\label{LHS}
%h(\dot\gamma,\dot{\mathcal P})=1+\mu\nu=\textrm{constant}.
%\end{equation}

\item Remark that we can always extend an $F$-ray $\mathcal P$ from $p$, i.e. a twisted meridian, beyond its initial point obtaining in this way an $F$-geodesic segment by twisting a similarly extended meridian. 
For any point $q\neq p$ in $M$ it is customary to denote by $\tau_q:[0,\infty)\to M$ be the unit speed $h$-geodesic emanating from $q=\tau_q(0)$ through $p=\tau_q(\rho)$, where $d_h(q,p)=\rho$.

In this way we can construct Finsler geodesic segments from a point $q\neq p$ to $p$ (see Figure \ref{F-geod from q to p}). Remark that we obtain the geodesic segment 
$\mathcal{P}_q^-:{[-\rho,0]}\to M$, $\mathcal{P}_q^-(s)=\varphi(s,\mu^-_{q^-}(s))=\varphi(s,\mu_{q^-}(-s))$
where we denote $\mu^-_{q^-}(s):=\mu_{q^-}(-s)$ the inverse oriented meridian from $p$ to $q$,  $\mu^-_{q^-}(-\rho)=q^-$,  $\mu^-_{q^-}(0)=p$, 
$\rho:=d_h(q,p)=d_F(q,p)$. Let us denote the $F$-geodesic  from $q$ through $p$ obtained in this way by $\omega_q:[0,\infty)\to M$, $\omega(s)=\varphi(s,\tau_{q^-}(s))$. We say that   $\omega_q$ is obtained by twisting $\tau_{q^-}$ by the flow of $W $ keeping the vertex $p$ fixed. 

\begin{figure}[h]%\label{fig 2}
\begin{center} \setlength{\unitlength}{1cm} 
\begin{picture}(10,4.5) 
\qbezier(1,2.5)(7.025,2.35)(9.05,2.5)
\qbezier(2,1)(7.5,2.5)(9.05,4.5)
\put(6,2){$p$}
\put(2.6,2.7){$q^-=\tau_{q^-}(0)$}
%\multiput(7,0.6)(7,2.6){13} {\line(0,1){0.1}} 
\multiput(3,1.3)(0,0.2){6} {\line(0,1){0.1}}
\put(3,0.7){$q$}
\put(9,2){$\tau_{q^-}$}
\put(9,4.6){$\omega_{q}$}
\put(5,2.42){\vector(1,0){0.1}}
\put(5,1.95){\vector(2,1){0.1}}
\put(3,2){\vector(0,-1){0.1}}
\put (4.5,1.4){$\mathcal{P}_q^-|_{[-\rho,0]}$}
 \end{picture} 
 \end{center}
 \caption{The $F$-geodesic from $q$ to $p$.}\label{F-geod from q to p}
 \end{figure}
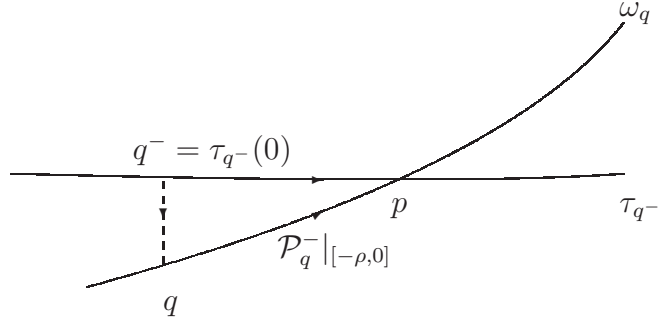 

\end{enumerate}

\end{remark}

%We denote by $q_c=\tau_q(s_0)$ the first conjugate point of $q$ along $\tau_q$. 

We will use in the following the naming $h$- and $F$-conjugate points for the conjugate points with respect to the Riemannian metric $h$ and the Finslerian metric $F$, respectively. Similarly, we will use $h$- and $F$-cut points for the cut points with respect to the Riemannian and Finslerian metric, respectively.

\begin{proof}[Proof of Theorem \ref{thm: Randers cut locus}]

First of all, observe that from our hypothesis we know that the $h$-cut locus of $q$ is exactly $\tau_q|_{[c,\infty)}$, where $\tau_q(c)$ is the first $h$-conjugate point of $q$ along $\tau_q$ (see Theorem 7.3.1 in \cite{SST}).

We divide our proof in two steps. 

At the first step, we will establish the correspondence of $h$-conjugate points of $q$ along $\tau_q$ with the $F$-conjugate points of $q$ along an $F$-geodesic from $q$.

Let $\tilde x=\tau_q(c)$ the first $h$-conjugate point of $q$ along $\tau_q$. Observe that in the case of the Riemannian surface of revolution $(M,h)$, we must have $c>\rho$, because $p$ is the unique pole for $h$. This is equivalent to saying that $\tilde x$ is conjugate to $q$ along $\tau_q$ (see \cite{SST}, \cite{T}).

Recall that $\tilde x=\tau_q(c)$ is the first $h$-conjugate point of $q$ along 
$\tau_q$ means that the Jacobi field along $\tau_q$ given by  
\begin{equation*}
Y_{q}(s)=\mathcal{M}_{a_1,\rho}(s)
\frac{\partial}{\partial \theta}|_{\tau_{q}},\quad s\in [\rho,\infty),
\end{equation*}
where $\mathcal{M}_{a_1,\rho}(s)$ is a smooth function along $\tau_{q}|_{[\rho,\infty)}$ depending on a constant $a_1$ chosen such that $m'$ is positive on $[0,a_1]$ and $\rho$. 

Moreover, if consider the vector field $J(s)$, along the twisted meridian $\mathcal R_q:[\rho,\infty)\to M$, $\mathcal R_q(s)=\varphi(s,\tau_q(s))$, defined by 
\begin{equation*}
J(s):=\varphi_{\tau_{q},*}(Y_{q}(s)),
\end{equation*}
then one can see that $J$ is actually a Jacobi field along $\mathcal R_q$. Indeed, one can easily verify that the flow $\varphi$ of $W$ maps the solutions of the Jacobi equation for $Y_{q}$ into the solutions of the Jacobi equation for $J(s)$, and therefore we have proved that the first $F$-conjugate point of $q$ is obtained at the intersection of the parallel through the first $h$-conjugate point with $\tau_q$.

At the second step, we will do the same thing for cut points of $q$, i.e. we will establish the correspondence of $h$-cut points of $q$ with the $F$-cut points of $q$. Namely,
 we will show that a point $\tilde{y}\in \tau_q|_{[c,\infty)}$ is an $h$-cut point of $q$ if and only if the point $y$, found at the intersection of the parallel through  $\tilde{y}$ with the twisted meridian $\{\varphi(s,\tau_q(s)):s\in {[c,\infty)}\}$ is an $F$-cut point of $q$. 

Indeed, such a $\tilde{y}$ is an $h$-cut point of $q$ if and only if there exists two $h$-geodesic segments $\alpha_1$ and $\alpha_2$ on $M$ from $q$ to  $\tilde{y}$ of equal $h$-length. By making use of Theorem 1.1 and an argument similar to Proposition 3, we can see that under the action of the flow $\varphi$ the end point  $\tilde{y}$ is clearly mapped into the point $y$ described above and the $h$-maximal geodesic segments $\alpha_1$ and $\alpha_2$ are deviated into two $F$-geodesic segments of same $F$-length from $q$ to $y$. This concludes the proof (see Figure \ref{fig: F-cut locus of q}).

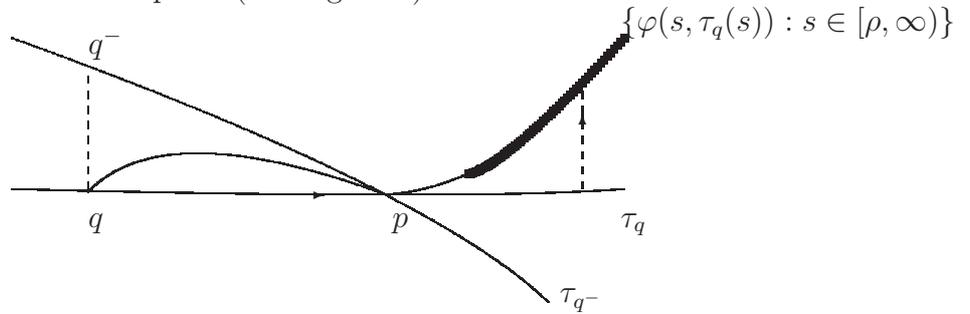
\begin{figure}[h]%\label{fig 2}
\begin{center} \setlength{\unitlength}{1cm} 
\begin{picture}(10,4.5) 
\qbezier(1,2.5)(7.025,2.35)(9.05,2.5)
\qbezier(1,4.5)(6.5,2.5)(8.05,1)
\qbezier(2,2.465)(3,3.5)(5.9,2.43)

\multiput(8.5,2.5)(0,0.2){7} {\line(0,1){0.1}}
\put(8.5,3.4){\vector(0,1){0.1}}
\put(6,2){$p$}
%\put(2.6,2.7){$q^-=\tau_{q^-_0}(0)$}
%\multiput(7,0.6)(7,2.6){13} {\line(0,1){0.1}} 
%\multiput(3,2.5)(0,0.2){6} {\line(0,1){0.1}}
\multiput(2,2.5)(0,0.2){8} {\line(0,1){0.1}}
\put(2,4.3){$q^-$}
%\put(3,3.9){$x^-$}
%\put(3,2){$x^-_0$}
%\put(2.5,3){$x_0$}
\put(2,2){$q$}
\put(9,2){$\tau_{q}$}
\put(9,4.6){$\{\varphi(s,\tau_q(s)):s\in {[\rho,\infty)}\}$}
\put(5,2.42){\vector(1,0){0.1}}
%\put(5,1.95){\vector(2,1){0.1}}
%\put(3,3.4){\vector(0,-1){0.1}}
%\put (4.5,1.4){$\mathcal{P}_q|_{[0,\rho]}$}
\put(8.2,1){$\tau_{q^-}$}
\qbezier(5.9,2.43)(7.5,2.5)(9.05,4.5)
%\put(7,3){$\varphi(t_0,\tau_q(t_0))$}
\linethickness{1mm} 
\qbezier(7,2.7)(7.5,2.8)(9.05,4.5)
%\qbezier(3,2.93)(4,3.2)(5.9,2.43)
\end{picture} 
\end{center}
\caption{The thick line is the $F$-cut locus of $q$.}\label{fig: F-cut locus of q}
\end{figure} 

\end{proof}
%%%%%%%%%%%%%%%%%%%%%%%%%%%%%%%%%%%%%%%%%%%%%

\noindent
KMITL, Bangkok, Thailand\\
E-mail: jimreivat99@gmail.com\\
E-mail: rattanasakhama@gmail.com\\

\bigskip

\noindent
Tokai University, Sapporo, Japan\\
E-mail: sorin@tokai.ac.jp

\end{document}